\documentclass[a4paper, reqno,12pt]{amsart}
\usepackage{tikz, tikz-3dplot}
\usepackage{amsmath, amsthm, amssymb}

\theoremstyle{plain}
\newtheorem{te}{Theorem}[section]
\newtheorem{lem}[te]{Lemma}

\newtheorem{pr}[te]{Proposition}

\newtheorem{con}[te]{Conjecture}
\theoremstyle{remark}
\newtheorem{re}[te]{Remark}
\newtheorem*{ack*}{Acknowledgment}

\def\y{{\bf y}}

\def\w{{\bf w}}
\def\u{{\bf u}}
\def\v{{\bf v}}
\def\l{{\bf l}}

\def\0{{\bf 0}}

\def\R{{\mathbb R}}

\def\C{{\mathbb C}}

\def\Z{{\mathbb Z}}
\def\A{{\mathbb A}}

\def\nint{\mathop{\diagup\kern-13.0pt\int}}

\def\les{{\;\lessapprox}\;}
\def\dist{{\operatorname{dist}\,}}

\def\Ic{{\mathcal I}}

\def\Mc{{\mathcal M}}\def\Ec{{\mathcal E}}
\def\Bc{{\mathcal B}}\def\Rc{{\mathcal R}}

\def\Pc{{\mathcal P}}

\def\Hc{{\mathcal H}}

\def\a{{\mathbb {a}}}

\begin{document}

\title[$L^{12}$ square root cancellation]{On $L^{12}$ square root cancellation for exponential sums associated with nondegenerate curves in $\R^4$}

\author{Ciprian Demeter}
\address{Department of Mathematics, Indiana University, 831 East 3rd St., Bloomington IN 47405}
\email{demeterc@indiana.edu}

\thanks{The author is  partially supported by the NSF grant DMS-1800305}

\begin{abstract}
We prove sharp $L^{12}$ estimates for exponential sums associated with nondegenerate curves in $\R^4$.  We place Bourgain's seminal result \cite{Bo} in a larger framework that contains a continuum of estimates of different flavor. We enlarge the spectrum of methods by combining decoupling with quadratic Weyl sum estimates, to address  new cases of interest. All results are proved in the general framework of real analytic curves.	
\end{abstract}

\maketitle

\section{Introduction}

Throughout this paper,  $\phi_3,\phi_4$ will be real analytic functions defined on some open interval containing $[\frac12,1]$, and  satisfying
\begin{equation}
\label{1}
\|\phi_k'\|_{C^3}=\sum_{1\le n\le 4}\max_{\frac12\le t\le 1}|\phi_k^{(n)}(t)|\le A_1, \;\;k\in\{3,4\},
\end{equation}

\begin{equation}
\label{2}
A_2\le \left|\det
\begin{bmatrix}
\phi_3^{(3)}(t)&\phi_3^{(4)}(t)\\ \phi_4^{(3)}(s)&\phi_4^{(4)}(s)
\end{bmatrix}\right|\le A_3,\;\;t,s\in[\frac12,1],
\end{equation}
\begin{equation}
\label{3}
|\phi_3^{(3)}(t)|\ge A_4,\;\; t\in [\frac12,1].
\end{equation}$A_1,\ldots,A_4$ are positive numbers that will determine the implicit constants in various inequalities.
While $\phi_3,\phi_4$ being $C^4$ would suffice for our purposes, we choose to work with real analytic functions for purely aesthetic reasons.
The examples of most immediate interest are power functions  $\phi_3(t)=t^{a}$, $\phi_4(t)=t^b$, with  the real numbers $a$ and $b$ satisfying the restrictions $a\not=b$ and  $a,b\not\in\{0,1,2\}$.
\smallskip
	
For a finite interval $I\subset \Z$, we write (ignoring the dependence on  $\phi_k$)	
$$\Ec_{I,N}(x)=\sum_{n\in I}e(nx_1+n^2x_2+\phi_3(\frac{n}{N})x_3+\phi_4(\frac{n}{N})x_4).$$
We make the following conjecture. 	
\begin{con}
	\label{51}	
	Assume $\alpha\ge \beta\ge 0$ and $\alpha+\beta=3$.	
	Let $\omega_3=[0,N^{\alpha}]$ and $\omega_4=[0,N^{\beta}]$.
	Assume $\phi_3,\phi_4$ are real analytic on $(0,3)$ and satisfy \eqref{1}, \eqref{2} and \eqref{3}.
	Then
	\begin{equation}
	\label{66}
	\int_{[0,1]\times [0,1]\times \omega_3\times \omega_4}|\Ec_{[\frac{N}2,N],N}(x)|^{12}dx\lesssim_\epsilon N^{9+\epsilon}.
	\end{equation} 	
\end{con}
For virtually all conceivable applications of \eqref{66}, both $\phi_3$ and $\phi_4$ will satisfy \eqref{3}. Because of this symmetry, our restriction $\alpha\ge \beta$ is essentially meaningless.
		
Let us write $A\lesssim B$ or $A=O(B)$ if $|A|\le CB$ for some, possibly large, but universal constant $C$. The notation $A\sim B$ will mean that $A\lesssim B$ and $B\lesssim A$ at the same time. We will write $A\ll B$ or $A=o(B)$ if $|A|\le cB$, for some small enough, universal constant $c$. The notation $A\les B$ will mean $A\lesssim (\log N)^{O(1)}B$, where $N$ will be the key scale parameter.
\smallskip

Since $|\Ec_{[\frac{N}{2},N],N}(x)|\sim N$ if $x\in[0,o(\frac1N)]\times [0,o(\frac1{N^2})]\times [0,o(1)]^2$, the exponent 9 is optimal in \eqref{66}. While the large value $N$ is attained by $|\Ec_{[\frac{N}{2},N],N}(x)|$ for a small subset of $x$, estimate \eqref{66} implies that the  average value of the exponential sum on the given domain  is $O(N^{\frac12+\epsilon})$, when measured in $L^{12}$. We call this $L^{12}$ square root cancellation. The relevance of the requirement $\alpha+\beta=3$ is that it guarantees
$$|[0,1]\times[0,1]\times \omega_3\times \omega_4|(\sqrt{N})^{12}=N^9.$$

The case $(\alpha,\beta)=(2,1)$ of the conjecture has been settled in \cite{Bo}. This inequality was proposed by Huxley \cite{Hu}, with $\phi_3(t)=t^{3/2}$, $\phi_4(t)=t^{1/2}$. In \cite{Bo}, it serves as the main ingredient in sharpening the record on the Lindel\"of hypothesis.

The only other known case prior to our work was $(\alpha,\beta)=(3,0)$. The variable $x_4$ and the function $\phi_4$ play no role in this case, as $\phi_4(\frac{n}{N})x_4=O(1)$. We treat $e(\phi_4(\frac{n}{N})x_4)$ as a coefficient $c_n=O(1)$.
In this case,  \eqref{66} follows from the inequality
$$\int_{[0,1]^2\times [0,N^3]}|\sum_{n\in I}c_ne(nx_1+n^2x_2+\phi_3(\frac{n}{N})x_3|^{12}dx_1dx_2dx_3\lesssim_{\epsilon}N^{9+\epsilon}\|c_n\|_{l^\infty}.$$
Assuming $\phi_3$ satisfies \eqref{1} and \eqref{3},  this was known as the Main Conjecture in Vinogradov's Mean Value Theorem, and was first solved in \cite{Wo}, then in \cite{BDG}.
This reduction fails for all other values  $\alpha<3$, since the length of the interval $\omega_3$ becomes too small. The proofs in \cite{Bo} and \cite{BDG} for $\alpha=2$ and $\alpha=3$
 are fairly different, in spite of both relying entirely on abstract decoupling.
\medskip

Our main result here verifies Conjecture \ref{51} in the range $\frac32\le \alpha< 2$.

\begin{te}
	\label{4linver}	
	Assume that $\frac32\le \alpha\le 2$. Assume $\phi_3,\phi_4$ are real analytic on $(0,3)$ and satisfy \eqref{1}, \eqref{2} and \eqref{3}.
	
	Then
	\begin{equation}
	\label{106}
	\int_{[0,1]\times [0,1]\times \omega_3\times \omega_4}|\Ec_{[\frac{N}{2},N],N}(x)|^{12}dx\lesssim_\epsilon N^{9+\epsilon}.
	\end{equation}	
\end{te}
If $\phi_3(t)=t^a$, $\phi_4(t)=t^b$ with $\alpha\le a$, $\beta\le b$, $a+b\le 9$, $a\not=b$ and  $a,b\not\in\{0,1,2\}$, a very simple rescaling argument for each member of the sum
$$\Ec_{[1,N],N}(x)=\sum_{1\le M\le N\atop{M\text{ dyadic}}}\Ec_{[\frac{M}{2},M],N}(x)$$
shows that \eqref{106} holds with $\Ec_{[\frac{N}{2},N],N}$ replaced with $\Ec_{[1,N],N}(x)$. In particular, this is always the case for the moment curve $\phi_3(t)=t^3$, $\phi_4(t)=t^4$. Other cases such as  $(\alpha,\beta)=(2,1)$, $(a,b)=(\frac32,\frac12)$ require slightly more sophisticated arguments similar to the one in \cite{Bo}, but will not be pursued here.
\smallskip

Theorem \ref{4linver} will follow from its bilinear analog. We prove this reduction  in Section \ref{oriruigui}.
\begin{te}
\label{4}	
Let $I_1,I_2$ be intervals of length $\sim N$ in $[\frac{N}{2},N]$, with $\dist(I_1,I_2)\sim {N}$. Assume $\phi_3,\phi_4$ are real analytic on $(0,2)$ and satisfy \eqref{1}, \eqref{2} and \eqref{3}. Assume that $\frac32\le \alpha\le 2$.

Then
$$\int_{[0,1]\times [0,1]\times \omega_3\times \omega_4}|\Ec_{I_1,N}(x)\Ec_{I_2,N}(x)|^6dx\lesssim_\epsilon N^{9+\epsilon}.$$
The implicit constant in this inequality is uniform over $A_1,\ldots,A_4\sim 1$.	
\end{te}

The reason we prove Theorem \ref{4linver} using bilinear, as opposed to linear or trilinear methods, is rather delicate.
As explained earlier, our results are sharp, the exponent 9 in \eqref{106} cannot be lowered. Each of the decoupling inequalities relevant to this paper has a certain critical exponent $p_c>2$. Experience shows that to achieve sharp results via decoupling, this tool must be used at the critical exponent $p_c$. When we apply decoupling on spatial balls of radius $M$, we decouple the curve into arcs of length $M^{-1/2}$. It seems very likely that the critical exponent for such a decoupling in the linear setting is larger than 12. See \cite{Hon} for a detailed discussion on this. Because of this, using linear $L^{12}$ decoupling turns out to be inefficient. Bilinearizing instead, gives us access to the $L^6$ decoupling of the parabola. This is an ideal scenario, since 6 is precisely the critical exponent in this setting.

The fact that $12=6\times 2$ turns out to be crucial to our argument. The other  factorization $12=4\times 3$ is also inefficient. An approach based on trilinearity would use $L^4$ estimates for hypersurfaces in $\R^4$. But the critical exponent here is $10/3$, not $4$.

Most of the paper is devoted to proving Theorem \ref{4}. The proof  will combine abstract decoupling methods with quadratic Weyl sum estimates. The decoupling techniques are introduced in Section \ref{s4}.
While these results are by now standard, the observation that the superficially stronger $l^2(L^{12})$  decoupling holds true for nondegenerate curves in the bilinear setting  appears to be new. One of the key features in our argument is the use of this inequality in places where the $l^{12}(L^{12})$  decoupling used in \cite{Bo} becomes inefficient. We combine this finer decoupling with estimates from number theory. In short, here is how our approach works. The initial integral involves quartic Weyl sums, for which sharp estimates are totally out of reach at the moment. Decoupling is applied once or twice in order to lower the complexity of the sums, to the level of manageable quadratic Weyl sums. These sums will appear in various combinations, and need to be tackled with extreme care, using various counting arguments such as Lemma \ref{42} and Lemma \ref{76}.

In Section \ref{Bsarg}, we start with a careful examination of Bourgain's argument from \cite{Bo}, for $\alpha=2$. In many ways, this case  turns out to be the easiest, as it works via just $l^{12}L^{12}$ decoupling, and  without any input from number theory. In Section \ref{s6} we introduce our new methodology, addressing the symmetric case $\alpha=\beta=\frac32.$ This ends up being the most delicate case, since it captures the biggest  region near the origin where constructive (and near constructive) interference occurs.
Also, it is in this case that the curve looks most genuinely four dimensional, as both $\omega_3$ and $\omega_4$ are large. For comparison, recall that when $\alpha=3$ the curve degenerates to a three dimensional one.
Sections \ref{s7} and \ref{s8} extend our method to the remaining cases, by successively building on each other. The case $\frac95\le \alpha<2$ combines elements of both approaches.

To the best of our knowledge, this paper represents the first systematic effort to combine abstract decoupling with Weyl sum estimates. The results proved here are part of the vast program initiated in \cite{DGW}, concerned with proving sharp $L^p$ estimates for the moment curve on spatial domains smaller than the torus. In \cite{DGW} only the moment curve in $\R^3$ is considered, and all estimates there rely solely on decoupling techniques.

There remain a lot of interesting related questions. One of them has to do with proving Conjecture \ref{51} in the range $2< \alpha<3$. We may speculate that the solution would combine some of the tools from our paper with those used to solve the case $\alpha=3$, see also Remark \ref{66345}. Second, $L^p$ moments are also worth investigating for smaller values $p<12$, in particular for $p=10$. See for example  \cite{Bo34} for some recent progress and some interesting applications. Section \ref{so} of our paper contains an  example that describes some of the enemies and limitations of square root cancellation in this setting. It seems plausible that small cap decoupling for the parabola (see \cite{DGW}) will be the right tool to attack this problem. We hope to address some of these questions in future work.        	

\begin{ack*}
	The author is grateful to Hongki Jung and Zane Li for pointing out a few typos in the earlier version of the manuscript.
\end{ack*}

\section{Decoupling for nondegenerate curves in $\R^4$}
\label{s4}

Let us start by recalling the decoupling for nondegenerate  curves in $\R^2$.
\begin{te}[\cite{BD3}]
\label{decpar}	
	Let $\phi:[0,1]\to\R$ be a $C^3$ function, with $\|\phi'\|_{C^2}=A_5<\infty$ and  $\min_{0\le t\le 1}|\phi''(t)|=A_6>0.$ Then for each $f:[0,1]\to\C$ and each ball $B_N\subset \R^2$ with radius $N$ we have
$$\|\int_{[0,1]}f(t)e(tu+\phi(t)w)dt\|_{L_{u,w}^6(B_N)}\lesssim_\epsilon N^\epsilon (\sum_{J\subset [0,1]\atop{|J|=N^{-1/2}}} \|\int_{J}f(t)e(tu+\phi(t)w)dt\|_{L_{u,w}^6(B_N)}^2)^{1/2}. $$		
The implicit constant is uniform over $A_5\sim 1$, $A_6\sim 1$.
\end{te}
We use this result to illustrate in the simplest terms the reduction of cubic terms used in the next section for $\alpha=2$. Namely, let us show that for $1\le \alpha<3$
\begin{equation}
\label{uy4tc674t7cyrfu}
\int_{[0,1]^2}|\sum_{1\le m\le M}e(mu+(m^2+\frac{m^3}{M^\alpha})w)|^6dudw\lesssim_\epsilon M^{3+\epsilon}.
\end{equation}
The term $\frac{m^3}{M^\alpha}w$ is not negligible (in the sense of Lemma \ref{22}), as it is not $O(1)$.
However, after a change of variables and using periodicity in $u$, we may rewrite the integral as
$$\frac1{M^4}\int_{[0,M^2]^2}|\sum_{1\le m\le M}e(\frac{m}{M}u+\phi(\frac{m}{M})w)|^6dudw,$$
where $\phi(t)=t^2+t^3M^{1-\alpha}$. Note that $A_5,A_6\sim 1$, uniformly over $M$. Inequality \eqref{uy4tc674t7cyrfu} is now a  standard consequence of Theorem \ref{decpar} with $N=M^2$. The cubic term  becomes a perturbation of the quadratic term, and does not significantly affect the constant $A_6$. Theorem \ref{65} will formalize this approach in four dimensions.
\smallskip

Throughout this section, $\phi_3,\phi_4$ are arbitrary functions satisfying \eqref{1} and \eqref{2}.
We denote by $E$ the extension operator associated with the curve $\Phi$
\begin{equation}
\label{curve}
\Phi(t)=(t,t^2,\phi_3(t),\phi_4(t)),\;t\in[\frac12,1].
\end{equation}
 More precisely, for $f:[\frac12,1]\to \C$ and $I\subset [\frac12,1]$ we write
$$E_If(x)=\int_If(t)e(tx_1+t^2x_2+\phi_3(t)x_3+\phi_4(t)x_4)dt.$$
The following $l^6(L^6)$ decoupling was  proved in \cite{Bo}, see also \cite{BD13}. It is in fact a bilinear version of  the  $l^{12}L^{12}$ decoupling for the curve \eqref{curve}.
\begin{te}
\label{48}
Let $I_1,I_2$ be two intervals of length $\sim 1$ in $[\frac12,1]$, with $dist(I_1,I_2)\sim 1$. Let also $f_i:[\frac12,1]\to \C$. Then for each ball $B_N$ of radius $N$ in $\R^4$ we have
$$\|E_{I_1}f_1E_{I_2}f_2\|_{L^6(B_N)}\lesssim_\epsilon N^{\frac13+\epsilon}(\sum_{J_1\subset I_1}\sum_{J_2\subset I_2}\|E_{J_1}f_1E_{J_2}f_2\|_{L^6(B_N)}^6)^{1/6}.$$
The sum on the right is over intervals $J_i$ of length $N^{-1/2}$.
\end{te}

In \cite{Bo}, this result is used in conjunction with the following estimate, an easy  consequence of transversality.
\begin{equation}
\label{52}
\|E_{J_1}f_1E_{J_2}f_2\|_{L^6(B_N)}^6\lesssim N^{-4}\|E_{J_1}f_1\|_{L^6(B_N)}^6\|E_{J_2}f_2\|_{L^6(B_N)}^6.
\end{equation}
This is inequality (13) in \cite{BD13}, and there is a detailed proof there. For reader's convenience, we sketch a somewhat informal  argument below.

The Fourier transform of $|E_{J_i}f_i|^6$ is supported inside a rectangular box with dimensions $\sim N^{-1}\times N^{-1}\times N^{-1}\times N^{-1/2}$.
We have the following wavepacket representation on $B_N$, slightly simplified for exposition purposes
$$|E_{J_i}f_i(x)|^6\approx \sum_{P_i\in\Pc_i}a_{P_i}1_{P_i}(x).$$
The coefficients $a_{P_i}$ are nonnegative reals.
The rectangular boxes $P_i$ have dimensions $\sim N\times N\times N\times N^{1/2}$ and tile $B_N$. They can be thought of as $N^{1/2}$-neighborhoods of cubes with diameter $\sim N$ inside  hyperplanes $\Hc_i$.
Since $\dist(J_1,J_2)\sim 1$, the angle between the normal vectors of $\Hc_i$ is $\sim 1$. Thus,  we have that $|P_1\cap P_2|\lesssim N^3$. We conclude by writing
\begin{align*}
\|E_{J_1}f_1E_{J_2}f_2\|_{L^6(B_N)}^6&\sim \sum_{P_1\in\Pc_1}\sum_{P_2\in\Pc_2}a_{P_1}a_{P_2}|P_1\cap P_2| \\&\lesssim N^{-4}\sum_{P_1\in\Pc_1}\sum_{P_2\in\Pc_2}a_{P_1}a_{P_2}|P_1||P_2| \\&\approx N^{-4}\|E_{J_1}f_1\|_{L^6(B_N)}^6\|E_{J_2}f_2\|_{L^6(B_N)}^6.
\end{align*}

It is worth observing that $\lesssim $ in inequality \eqref{52} is essentially an (approximate) similarity $\approx$,  making \eqref{52} extremely efficient. Indeed, since $P_1$ and $P_2$ intersect $B_N$, we have that $|P_1\cap P_2|\sim N^3$.
\medskip

To address new values of $\alpha$ in this paper, we will need the following $l^2(L^6)$ decoupling. This implies the previous $l^6(L^6)$ decoupling, and provides a critical improvement in the cases when the terms in the sum are of significantly different sizes.
\begin{te}
	\label{49}
	Let $I_1,I_2$ be two intervals of length $\sim 1$ in $[\frac12,1]$, with $dist(I_1,I_2)\sim 1$. Let also $f_i:[\frac12,1]\to \C$. Then for each ball $B_N$ of radius $N$ in $\R^4$ we have
	$$\|E_{I_1}f_1E_{I_2}f_2\|_{L^6(B_N)}\lesssim_\epsilon N^{\epsilon}(\sum_{J_1\subset I_1}\sum_{J_2\subset I_2}\|E_{J_1}f_1E_{J_2}f_2\|_{L^6(B_N)}^2)^{1/2}.$$
	The sum on the right is over intervals $J$ of length $N^{-1/2}$.
\end{te}
\begin{proof}
We briefly sketch the argument, that follows closely the one in \cite{Bo}. Let $b(N)$ be the best constant such that
	$$\|E_{I_1}f_1E_{I_2}f_2\|_{L^6(B_N)}\le b(N)(\sum_{J_1\subset I_1}\sum_{J_2\subset I_2}\|E_{J_1}f_1E_{J_2}f_2\|_{L^6(B_N)}^2)^{1/2}$$
	holds for all functions and balls as above. Fix $B_N$ and let $\Bc$ be a finitely overlapping cover of $B_N$ with balls $\Delta$ of radius $N^{2/3}$. It will soon become clear that the exponent $2/3$ is chosen in order to make cubic terms negligible.  Applying this inequality on each $\Delta$, then summing up using Minkowski's inequality shows that
	$$\|E_{I_1}f_1E_{I_2}f_2\|_{L^6(B_N)}\le b(N^{2/3})(\sum_{H_1\subset I_1}\sum_{H_2\subset I_2}\|E_{H_1}f_1E_{H_2}f_2\|_{L^6(B_N)}^2)^{1/2}.$$
	The intervals $H$ have length $N^{-1/3}$. We next analyze each term in the sum. Let $l_i$ be the left endpoint of $H_i$, and write a generic point  $t_i\in H_i$ as $t_i=l_i+s_i$, with $s_i\in [0,N^{-1/3}]$. We use  Taylor's formula for $k\in\{3,4\}$
	$$\phi_k(t_i)=\phi_k(l_i)+\phi_k'(l_i)s_i+\frac{\phi_{k}''(l_i)}{2}s_i^2+\psi_{k,i}(s_i),$$
	where, due to \eqref{1},  $\|\psi_{k,i}\|_{L^{\infty}([0,N^{-1/3}])}=O(\frac1N)$.
Let us write
	$$\begin{cases}y_1=x_1+2l_1x_2+\phi_3'(l_1)x_3+\phi_4'(l_1)x_4\\y_2=x_1+2l_2x_2+\phi_3'(l_2)x_3+\phi_4'(l_2)x_4\\y_3=x_2+\frac{\phi_3''(l_1)}{2}x_3+\frac{\phi_4''(l_1)}{2}x_4 \\y_4=x_2+\frac{\phi_3''(l_2)}{2}x_3+\frac{\phi_4''(l_2)}{2}x_4
	\end{cases}.
	$$	
	It follows that
	$$|E_{H_1}f_1(x)E_{H_2}f_2(x)|=$$$$|\int_{[0,N^{-1/3}]^2}f_1(l_1+s_1)e(s_1y_1+s_1^2y_3)f_2(l_2+s_2)e(s_2y_2+s_2^2y_4)e(L(y,s_1,s_2))ds_1ds_2|.$$
	Here
	$$L(y,s_1,s_2)= x_3(\psi_{3,1}(s_1)+\psi_{3,2}(s_2))+x_4(\psi_{4,1}(s_1)+\psi_{4,2}(s_2)).$$
	Lemma \ref{18} shows that $x_3,x_4$ depend linearly on $y_1,y_2,y_3,y_4$ with coefficients $O(1)$. This allows us to  write
	$$e(L(y,s_1,s_2))=e(\sum_{i=1}^4y_i(g_i(s_1)+h_i(s_2)))$$
	with $\|g_i\|_{\infty},\|h_i\|_{\infty}=O(\frac1N)$. Letting
	 $\bar{f}_i(s_i)=f_i(l_i+s_i)$ and
	 $$\begin{cases}\eta_1(s_1,s_2)=s_1+g_1(s_1)+h_1(s_2)\\ \eta_3(s_1,s_2)=s_1^2+g_3(s_1)+h_3(s_2)\\\eta_2(s_1,s_2)=s_2+g_2(s_1)+h_2(s_2)\\ \eta_4(s_1,s_2)=s_2^2+g_4(s_1)+h_4(s_2)\end{cases}$$
	 we write
	$$|E_{H_1}f_1(x)E_{H_2}f_2(x)|=$$$$|\int_{[0,N^{-1/3}]^2}\bar{f}_1(s_1)e(y_1\eta_1(s_1,s_2)+y_3\eta_{3}(s_1,s_2))\bar{f}_2(s_2)e(y_2\eta_2(s_1,s_2)+y_4\eta_4(s_1,s_2))ds_1ds_2|.$$
	For $\bar{J}_i\subset [0,N^{-1/3}]$ we write
	$$\Ic_{\bar{J}_1,\bar{J}_2}(y)=$$$$|\int_{\bar{J}_1\times \bar{J}_2}\bar{f}_1(s_1)e(y_1\eta_1(s_1,s_2)+y_3\eta_{3}(s_1,s_2))\bar{f}_2(s_2)e(y_2\eta_2(s_1,s_2)+y_4\eta_4(s_1,s_2))ds_1ds_2|.$$
	This is the extension operator associated with the surface $(\eta_1,\ldots,\eta_4)$, applied to the function $f_1\otimes f_2$.
	
	We use Lemma \ref{18} to write
	$$\int_{B_N}|E_{J_1}f_1(x)E_{J_2}f_2(x)|^6dx=\int_{\bar{B}_N}\Ic_{\bar{J}_1, \bar{J}_2}(y)^6dy.$$
	Here $\bar{B}_N$ is a ball of radius $\sim N$ and $J_i=\bar{J}_i+l_i$.
	Note that the surface $(\eta_1,\ldots,\eta_4)$ is within $O(N^{-1})$ from the surface
	$$(s_1,s_1^2,s_2,s_2^2),\;\;s_i\in [0,N^{-1/3}],$$
	so -for decoupling purposes- the two surfaces are indistinguishable when paired with spatial variables $y$ ranging through a ball of radius $N$. The latter surface admits an $l^2(L^6)$ decoupling, as can be easily seen by using Theorem \ref{decpar} twice. The same remains true for the surface  $(\eta_1,\ldots,\eta_4)$, and thus
	$$\|\Ic_{[0,N^{-1/3}],[0,N^{-1/3}]}\|_{L^6({\bar{B}_N})}\lesssim_\epsilon
	N^{\epsilon}(\sum_{{\bar{J}_1,\bar{J}_2}\subset [0,N^{-1/3}]}\|\Ic_{\bar{J}_1,\bar{J}_2}\|^2_{L^6({\bar{B}_N})})^{1/2}.$$
	The sum on the right is over intervals of length $N^{-1/2}$.
	If we undo the change of variables we find
	
$$\|E_{H_1}f_1E_{H_2}f_2\|_{L^6(B_N)}\lesssim_{\epsilon}N^{\epsilon}(\sum_{J_1\subset H_1}\sum_{J_2\subset H_2}\|E_{J_1}f_1E_{J_2}f_2\|^2_{L^6(B_N)})^{1/2}.$$
Putting things together, we have proved the bootstrapping inequality
$$b(N)\lesssim_\epsilon N^{\epsilon}b(N^{2/3}).$$
We conclude that $b(N)\lesssim_\epsilon N^{\epsilon}$, as desired.	

\end{proof}

We will also record the following close relative of Theorem \ref{49}, that will be needed in the next sections.

\begin{te}
\label{65}	
Assume $\psi_1,\ldots,\psi_4:[-1,1]\to\R$ have $C^3$ norm $O(1)$, and in addition satisfy
$$|\psi_2''(t)|,|\psi_3''(t)|\ll 1,\; \forall\; |t|\le 1$$
and
$$|\psi_1''(t)|,|\psi_4''(t)|\sim 1,\; \forall\; |t|\le 1.$$
 Let $E$ be the extension operator associated with the surface
$$\Psi(\xi_1,\xi_2)=(\xi_1,\xi_2,\psi_1(\xi_1)+\psi_2(\xi_2),\psi_3(\xi_1)+\psi_4(\xi_2)),\;\;|\xi_1|,|\xi_2|\le 1.$$
More precisely, for $F:[-1,1]^2\to\C$, $R\subset [-1,1]^2$ and $x\in \R^4$ we write
$$E_RF(x)=\int_{R}F(\xi_1,\xi_2)e(x\cdot \Psi(\xi_1,\xi_2))d\xi_1d\xi_2.$$
Then for each ball $B_N\subset \R^4$ with radius $N$ we have
\begin{equation}
\label{64}
\|E_{[-1,1]^2}F\|_{L^6(B_N)}\lesssim_\epsilon N^{\epsilon}(\sum_{H_1,H_2\subset [-1,1]}\|E_{H_1\times H_2}F\|^2_{L^6(B_N)})^{1/2},
\end{equation}
where the sum is taken over intervals of length $N^{-1/2}$.

In particular, for each constant coefficients $c_{m_1,m_2}\in\C$ we have
\begin{equation}
\label{59}
\|\sum_{m_1\le N^{1/2}}\sum_{m_2\le N^{1/2}}c_{m_1,m_2}e(x\cdot \Psi(\frac{m_1}{N^{1/2}},\frac{m_2}{N^{1/2}}))\|_{L^6(B_N)}\lesssim_{\epsilon}N^\epsilon \|c_{m_1,m_2}\|_{l^2}|B_N|^{1/6},
\end{equation}
while if $M\ge N^{1/2}$ we have
\begin{equation}
\label{80}
\|\sum_{m_1\le M}\sum_{m_2\le M}c_{m_1,m_2}e(x\cdot \Psi(\frac{m_1}{M},\frac{m_2}{M}))\|_{L^6(B_N)}
\end{equation}
$$\lesssim_{\epsilon}N^\epsilon (\sum_{J_1,J_2}\|\sum_{m_1\in J_1}\sum_{m_2\in J_2}c_{m_1,m_2}e(x\cdot \Psi(\frac{m_1}{M},\frac{m_2}{M}))\|^2_{L^6(B_N)})^{1/2},$$
with $J_i$ intervals of length $\sim MN^{-1/2}$ partitioning $[1,M]$.

The implicit constants in both inequalities are independent of $N$, $M$ and of $\psi_i$.
\end{te}
\begin{proof}
The exponential sum estimates \eqref{59}, \eqref{80} are standard consequences of \eqref{64}, so we will focus on proving the latter. 	
When $$\Psi(\xi_1,\xi_2)=(\xi_1,\xi_2,C_1\xi_1^2+C_2\xi_2^2,C_3\xi_1^2+C_4\xi_2^2)$$ with $|C_1|,|C_4|\sim 1$ and $|C_2|,|C_3|\ll 1$, the result follows by applying $l^2(L^6)$ Theorem \ref{decpar} twice (after an initial affine change of variables that reduces it to the case $C_1=C_4=1$, $C_2=C_3=0$).

We use a bootstrapping argument similar to the one in Theorem \ref{49}. Let $d(N)$	be the smallest constant in \eqref{64}. We need to prove that $d(N)\lesssim_\epsilon N^\epsilon$.

We first note that
$$\|E_{[-1,1]^2}F\|_{L^6(B_N)}\le d(N^{2/3})(\sum_{U_1,U_2\subset [-1,1]}\|E_{U_1\times U_2}F\|^2_{L^6(B_N)})^{1/2},$$
where $U_i$ are intervals of length $N^{-1/3}$ centered at $l_i$. When $|t|<\frac12N^{-1/3}$ we have
$$\psi_i(t)=\psi_i(l_i)+\psi_i'(l_i)t+C_it^2+O(\frac1N),$$
where $|C_1|,|C_4|\sim 1$ and $|C_2|,|C_3|\ll 1$. It follows that, after an affine change of variables, the restriction of $\Psi$ to $U_1\times U_2$ may be parametrized as
$$(\xi_1,\xi_2,C_1\xi_1^2+C_2\xi_2^2+O(\frac1N),C_3\xi_1^2+C_4\xi_2^2+O(\frac1N)),\;\;|\xi_1|,|\xi_2|=O(N^{-1/3}).$$
We decouple this on $B_N$, using the observation at the beginning of the proof. We find
$$\|E_{U_1\times U_2}F\|_{L^6(B_N)}\lesssim_{\epsilon}N^\epsilon (\sum_{H_1\subset U_1,H_2\subset U_2}\|E_{H_1\times H_2}F\|^2_{L^6(B_N)})^{1/2}.$$
It follows that $d(N)\lesssim_\epsilon N^\epsilon d(N^{2/3})$, which forces $d(N)\lesssim_\epsilon N^\epsilon$.

\end{proof}

\section{Bourgain's argument for the case $\alpha=2$}
\label{Bsarg}
For the remainder of the paper, we will use the following notation

$$\Ec_{I}(x)=\sum_{n\in I}e(N\Phi(\frac{n}{N})x)=\sum_{n\in I}e(nx_1+\frac{n^2}{N}x_2+\phi_3(\frac{n}{N})Nx_3+\phi_4(\frac{n}{N})Nx_4).$$
Note that, compared to $\Ec_{I,N}$, we dropped the subscript $N$ and renormalized the variables $x_2$, $x_3$ and $x_4$.
Letting $$\Omega=[0,N]^3\times[0,1]$$
and using periodicity in $x_1$,
we need to prove that
$$\int_{\Omega}|\Ec_{I_1}\Ec_{I_2}|^6\lesssim_\epsilon N^{9+\epsilon}.$$

Bourgain's argument from \cite{Bo} for the case $\alpha=2$ of Theorem \ref{4} involves three successive decouplings. We simplify it slightly it and reduce it to only two decouplings.
\\
\\
Step 1. Cover $\Omega$ with cubes $B$ of side length $1$,  apply $l^6(L^6)$  decoupling (Theorem \ref{48})  on each $B$ (or rather $NB$, after rescaling), then sum these estimates to get
$$\int_{\Omega}|\Ec_{I_1}\Ec_{I_2}|^6\lesssim_\epsilon N^{2+\epsilon}\sum_{J_1\subset I_1}\sum_{J_2\subset I_2}\sum_{B\subset \Omega}\int_{B}|\Ec_{J_1}\Ec_{J_2}|^6.$$
Here $J_1,J_2$ are intervals of length $N^{1/2}$.

The remaining part of the argument will show the uniform estimate $O(N^{6+\epsilon})$ for each term  $\sum_{B\subset \Omega}\int_{B}|\Ec_{J_1}\Ec_{J_2}|^6.$ Fix $J_i=[h_i,h_i+N^{1/2}]$.
\\
\\
Step 2. Note that when $x\in \Omega$
$$|\Ec_{J_i}(x)|=|\sum_{m\le N^{1/2}}c_m(x_4)e(mu_i+m^2w_i+\eta_i(m)x_3)|$$
where
\begin{equation}
\label{57}\begin{cases}u_i=x_1+\frac{2h_i}{N}x_2+\phi_3'(\frac{h_i}{N})x_3\\w_i=\frac{x_2}{N}+\phi_3''(\frac{h_i}{N})\frac{x_3}{2N}\\\eta_i(m)={m^3}\frac{\phi_3'''(\frac{h_i}{N})}{3!N^2}+{m^4}\frac{\phi_3''''(\frac{h_i}{N})}{4!N^3}+\ldots\end{cases}.
\end{equation}
We hide  the whole contribution from $x_4$ into the coefficient $c_{m}(x_4)$.
Indeed, since
$$\phi_4(\frac{h_i+m}{N})Nx_4=\phi_4(\frac{h_i}{N})Nx_4+m\phi_4'(\frac{h_i}{N})x_4+O(1),$$
$x_4$  does not contribute significantly with quadratic or higher order terms, so it produces no cancellations.  We will only use that $c_{m}(x_4)=O(1)$.

At this point, we  seek a change of variables. We want the new domain of integration to be a rectangular box, to allow us to separate the four-variable integral $\int_{\Omega}|\Ec_{J_1}\Ec_{J_2}|^6$ into the product of two-variable integrals. Note that the ranges of $x_1,x_2,x_3$ are the same, $[0,N]$, but $x_4$ is restricted to the smaller interval $[0,1]$. We cannot use periodicity to extend the range of $x_4$ to $[0,N]$, because the individual waves $e(m\phi_4'(\frac{h_i}{N})x_4)$ have different periods with respect to the variable $x_4$.
Because of this, the variable $x_4$ is practically useless from this point on, it will not generate oscillations. To generate a fourth variable with range $[0,N]$ for the purpose of a change of coordinates, Bourgain produces a piece of magic.

First, he applies \eqref{52} on each cube $NB$
\begin{align*}
\int_{B}|\Ec_{J_1}\Ec_{J_2}|^6&=N^{-4}\int_{NB}|\Ec_{J_1}(\frac{\cdot}{N})\Ec_{J_2}(\frac{\cdot}{N})|^6\\&\lesssim N^{-8}\int_{NB}|\Ec_{J_1}(\frac{\cdot}{N})|^6\int_{NB}|\Ec_{J_2}(\frac{\cdot}{N})|^6=\int_{B}|\Ec_{J_1}|^6\int_{B}|\Ec_{J_2}|^6.
\end{align*}
Second, he uses the following abstract inequality, that only relies on the positivity of $|\Ec_{J_i}|^6$
$$\sum_{B\subset \Omega}\int_{B}|\Ec_{J_1}|^6\int_{B}|\Ec_{J_2}|^6\lesssim \int_{\Omega}dx\int_{(y,z)\in [-1,1]^4\times [-1,1]^4}|\Ec_{J_1}(x+y)\Ec_{J_2}(x+z)|^6dydz.$$
Using periodicity in the $y_1,z_1$ variables, this is
$$\lesssim \frac1{N^2}\int_{x_1\in [0,N]\atop_{x_4,y_2,y_3,y_4,z_2,z_3,z_4\in[-1,1]}}dx_1\ldots dz_4\int_{y_1,z_1,x_2,x_3\in[0,N]}|\Ec_{J_1}(x+y)\Ec_{J_2}(x+z)|^6dy_1dz_1dx_2dx_3.$$
In short, the variable $x_1$ is now replaced with the new variables $y_1$ and $z_1$. It remains to prove that
\begin{equation}
\label{60}
\int_{y_1,z_1,x_2,x_3\in[0,N]}|\Ec_{J_1}(x+y)\Ec_{J_2}(x+z)|^6dy_1dz_1dx_2dx_3\lesssim_\epsilon N^{7+\epsilon},
\end{equation}
uniformly over $x_1,x_4,y_2,y_3,y_4,z_2,z_3,z_4$. With these variables fixed,
we make the affine change of variables $(y_1,z_1,x_2,x_3)\mapsto (u_1,u_2,w_1,w_2)$

\begin{equation}
\label{999999}
\begin{cases}u_1=(y_1+x_1)+\frac{2h_1}{N}(x_2+y_2)+\phi_3'(\frac{h_1}{N})(x_3+y_3)\\u_2=(z_1+x_1)+\frac{2h_2}{N}(x_2+z_2)+\phi_3'(\frac{h_2}{N})(x_3+z_3) \\w_1=\frac{x_2+y_2}{N}+\phi_3''(\frac{h_1}{N})\frac{x_3+y_3}{2N}\\w_2=\frac{x_2+z_2}{N}+\phi_3''(\frac{h_2}{N})\frac{x_3+z_3}{2N}\end{cases}.
\end{equation}
The  Jacobian is $\sim \frac1{N^2}$, due to \eqref{3}.
Note that $x_3+y_3=A(w_1-w_2)$, where $A$ depends just on $h_1,h_2$, and $|A|\sim N$. Using \eqref{57}   we may write the last integral as
$$
N^2\times
$$
\begin{equation}
\label{103}
\int_{|u_i|\lesssim N\atop_{|w_i|\lesssim 1}}|\sum_{m_i\le N^{\frac12}}c_{m_1,m_2}e(m_1u_1+m_1^2w_1+m_2u_2+m_2^2w_2+(\eta_1(m_1)+\eta_2(m_2))A(w_1-w_2))|^6.
\end{equation}
The coefficient $c_{m_1,m_2}$ depends only on $m_1,m_2,x_4, y_4,z_4$, but not on the variables of integration $u_i,w_i$.
The argument of each exponential  may be rewritten as
$$
\frac{m_1}{N^{1/2}}u_1N^{1/2}+(\psi_1(\frac{m_1}{N^{1/2}})+\psi_2(\frac{m_2}{N^{1/2}}))w_1N+$$$$\frac{m_1}{N^{1/2}}u_2N^{1/2}+(\psi_3(\frac{m_1}{N^{1/2}})+\psi_4(\frac{m_2}{N^{1/2}}))w_2N$$
where
$$\begin{cases}\psi_1(\xi)=\xi^2+&{\xi^3}\frac{A\phi_3'''(\frac{h_1}{N})}{3!N^{3/2}}+{\xi^4}\frac{A\phi_3''''(\frac{h_1}{N})}{4!N^{2}}+\ldots
\\\psi_2(\xi)=&{\xi^3}\frac{A\phi_3'''(\frac{h_2}{N})}{3!N^{3/2}}+{\xi^4}\frac{A\phi_3''''(\frac{h_2}{N})}{4!N^{2}}+\ldots\\\psi_3(\xi)=&{-\xi^3}\frac{A\phi_3'''(\frac{h_1}{N})}{3!N^{3/2}}-{\xi^4}\frac{A\phi_3''''(\frac{h_1}{N})}{4!N^{2}}-\ldots
\\\psi_4(\xi)=\xi^2-&{\xi^3}\frac{A\phi_3'''(\frac{h_2}{N})}{3!N^{3/2}}-{\xi^4}\frac{A\phi_3''''(\frac{h_2}{N})}{4!N^{2}}-\ldots\end{cases}.$$
These functions satisfy the requirements in Theorem \ref{65}. The integral \eqref{103} is the same as
$$N^{-3}\int_{u_i=O(N^{3/2})\atop_{w_i=O(N)}}|\sum_{m_1\le N^{1/2}}\sum_{m_2\le N^{1/2}}c_{m_1,m_2}e((u_1,u_2,w_1,w_2)\cdot\Psi(\frac{m_1}{N^{1/2}},\frac{m_2}{N^{1/2}}))|^6du_1du_2dw_1dw_2.$$
If we cover the domain of integration with $\sim N$ balls $B_N$ and apply \eqref{59} on each of them, we may dominate the above expression by
$$N^{-3}N(N^{\frac12+\epsilon}N^{4/6})^6=N^{5+\epsilon}.$$
This proves \eqref{60} and ends the proof. Note that this argument treats the cubic and higher order terms as perturbations of quadratic factors, as explained in the proof of \eqref{uy4tc674t7cyrfu}.
\smallskip

In summary, what is special about the case $\alpha=2$ is that the  range of $x_3$ in our initial integral over $\Omega$ is $[0,N]$. This was needed in producing the large spatial range $w_{i}=O(N)$ for our final variables, crucial for the application of \eqref{59}. This inequality provides decoupling into point masses, reducing the initial exponential sum to individual waves. In Section \ref{s8} we will see that when $\alpha$ is slightly smaller than 2, inequality \eqref{80} will have to replace \eqref{59}, leading to quadratic Weyl sums whose handling demands number theory.
\begin{re}
\label{66345}	
It is not clear whether a version of Bourgain's method could be made to work in the range $2<\alpha<3$. If successful, this would potentially provide a new argument for Vinogradov's Mean Value Theorem in $\R^3$. 
Decoupling on cubes with size $N^{\beta-1}$ and using \eqref{52} on balls $B_{N^\beta}$ leads to variables $y_1,z_1$ with associated  period equal to 1, much bigger than their range $N^{\beta-1}$. The change of variables \eqref{999999} is no longer efficient in this case. 
\end{re}

\section{Proof of Theorem \ref{4} in the case $\alpha=\frac32$}
\label{s6}
This time we let $\Omega=[0,1]\times [0,N]\times [0,N^{1/2}]\times [0,N^{1/2}]$. Recall that
$$\Ec_I(x)=\sum_{n\in I}e(nx_1+\frac{n^2}Nx_2+\phi_3(\frac{n}{N})Nx_3+\phi_4(\frac{n}{N})Nx_4).$$
We need to prove
\begin{equation}
\label{74}
\int_{\Omega}|\Ec_{I_1}\Ec_{I_2}|^6dx\lesssim_\epsilon N^{8+\epsilon}.
\end{equation}
In this case, we will only need assumptions \eqref{1} and \eqref{2}, but not \eqref{3}.
We start by presenting a general principle that will explain the subtleties of our argument. See also Remark \ref{69}.
\medskip

Consider two partitions of $\Omega$, one into cubes $B$ with side length $l$ and another one into cubes $\Delta$ with side length $L\ge l$. The intervals $J_i$ have length $\sqrt\frac{N}{l}$ and partition $I_i$. The intervals $U_i$ have length $\sqrt\frac{N}{L}$ and partition $I_i$.  The following holds, via two applications of Theorem \ref{49} (on cubes $B$ and $\Delta$, combined with Minkowski's inequality)
\begin{equation}
\label{72}
\|\Ec_{I_1}\Ec_{I_2}\|_{L^6(\Omega)}\lesssim_\epsilon N^{\epsilon}(\sum_{J_1,J_2}\|\Ec_{J_1}\Ec_{J_2}\|_{L^6(\Omega)}^2)^{1/2}\lesssim_\epsilon N^{\epsilon}(\sum_{U_1,U_2}\|\Ec_{U_1}\Ec_{U_2}\|_{L^6(\Omega)}^2)^{1/2}.
\end{equation}
Also, combining the above inequalities with H\"older shows that
\begin{equation}
\label{73}
\|\Ec_{I_1}\Ec_{I_2}\|_{L^6(\Omega)}\lesssim_\epsilon$$$$ N^{\epsilon}(\sharp(J_1,J_2))^{\frac13}(\sum_{J_1,J_2}\|\Ec_{J_1}\Ec_{J_2}\|_{L^6(\Omega)}^6)^{1/6}\lesssim_\epsilon N^{\epsilon}(\sharp(U_1,U_2))^{\frac13}(\sum_{U_1,U_2}\|\Ec_{U_1}\Ec_{U_2}\|_{L^6(\Omega)}^6)^{1/6}.
\end{equation}

Invoking periodicity in $x_1,x_2$ and the invariance of \eqref{1}, \eqref{2}, \eqref{3} under the change of sign $\phi_k\mapsto -\phi_k$, \eqref{74} is equivalent with proving that
\begin{equation}
\label{45}
\int_{[-N^{1/2},N^{1/2}]\times[-N,N]\times [-N^{1/2},N^{1/2}]\times [-N^{1/2},N^{1/2}]}|\Ec_{I_1}(x)\Ec_{I_2}(x)|^6dx\lesssim_\epsilon N^{8+\frac12+\epsilon}.
\end{equation}

We first demonstrate the inefficiency of $l^6L^6$ decoupling for this case, by working with the smaller domain
$$S=[-o(N^{1/2}),o(N^{1/2})]^4.$$
 We cover $S$ with unit cubes and apply decoupling into intervals $J_1,J_2$  of length $N^{1/2}$ as in the previous section,  to dominate
\begin{equation}
\label{46}
\int_S|\Ec_{I_1}(x)\Ec_{I_2}(x)|^6dx\lesssim_\epsilon N^{\epsilon}N^{6(\frac12-\frac16)}\sum_{J_1,J_2}\int_{S}|\Ec_{J_1}(x)\Ec_{J_2}(x)|^6dx.
\end{equation}
We will next show that the right hand side is too big, thus leading to an overestimate for our initial integral.
 When  $ J=[h,h+N^{1/2}]$ and $n=h+m\in J$ we write
 $$\phi_k(\frac{n}{N})=\phi_k(\frac{h}{N})+\phi_k'(\frac{h}{N})\frac{m}{N}+\frac12\phi_k''(\frac{h}{N})(\frac{m}{N})^2+O(\frac{m}{N})^3.$$
 If $|x_3|, |x_4|\ll N^{1/2}$ and $m\le N^{1/2}$, we guarantee that the contribution from higher order terms is small
 $$O(\frac{m}{N})^3N(|x_3|+|x_4|)\ll 1.$$
 If we collect the contributions from linear and quadratic terms we find
$$|\Ec_J(x)|=|\sum_{m\le N^{1/2}}e(mu+m^2w+o(1))|$$
where
$$
\begin{cases}u=x_1+\frac{h}{N}x_2+\phi_3'(\frac{h}{N})x_3+\phi_4'(\frac{h}{N})x_4\\w=\frac{x_2}{N}+\frac12\phi_3''(\frac{h}{N})\frac{x_3}{N}+\frac12\phi_4''(\frac{h}{N})\frac{x_4}{N}\end{cases}.
$$
Using Lemma \ref{18} we write
$$\int_{S}|\Ec_{J_1}(x)\Ec_{J_2}(x)|^6dx\gtrsim $$$$N^2\int_{(u_1,u_2)\in [0,o(N^{1/2})]^2\atop{(w_1,w_2)\in [0,o(N^{-1/2})]^2}}|\sum_{m\le N^{1/2}}e(mu_1+m^2w_1+o(1))|^6|\sum_{m\le N^{1/2}}e(mu_2+m^2w_2+o(1))|^6.$$
We now use the fact that we have constructive interference
$$|\sum_{m\le N^{1/2}}e(mu+m^2w+o(1))|\sim N^{1/2}$$
on the set of measure $\sim \frac1N$ $$(u,w)\in (\bigcup_{l\in \{0,1,\ldots,o(\sqrt{N})\}}[l,l+\frac1{\sqrt{N}}])\times [0,\frac1{N}].$$
It follows that
$$\int_{S}|\Ec_{J_1}(x)\Ec_{J_2}(x)|^6dx\gtrsim N^2N^{6}N^{-2}=N^6.$$
It is not hard to prove that this lower bound is sharp, but this has no relevance to us here.
The point of working with the symmetric domain $S$ was to make sure that $w_1,w_2\sim \frac1N$ are in the new domain of integration. Going back to \eqref{46}, the $l^6(L^6)$ decoupling method leads to  the upper bound
$$\int_S|\Ec_{I_1}(x)\Ec_{I_2}(x)|^6dx\lesssim_\epsilon N^{9+\epsilon}.$$
This falls short by the factor $N^{1/2}$ from proving \eqref{45}.

The second inequality in \eqref{73} shows that using $l^6(L^6)$ decoupling on  cubes $\Delta$ that are larger than $N$ will only worsen the upper bounds we get. On the other hand, working with smaller cubes will render decoupling inefficient. The resulting  exponential sums will be very difficult to handle using number theory, since the cubic terms are no longer $O(1)$ in this case.
\smallskip

Let us now describe the correct approach, that will critically rely on $l^2$, rather than $l^6$ decoupling. The following level set estimate will play a key role in various counting arguments. The main strength of the lemma is in the case when $|l_1|\sim |l_2|$.

Throughout the remainder of the paper, the letter $l$ will be used to denote integers, and their relative proximity to powers of $2$ will be denoted using the symbol $\sim$.  We make the harmless  convention to write $0\sim 2^0$.

\begin{lem}
	 	\label{42}
	 	Assume $\phi_3$, $\phi_4$ satisfy \eqref{1} and  \eqref{2}. Let $l_1,l_2$ with $\max\{|l_1|,|l_2|\}\sim 2^{j}$, $j\ge 0$, and let $$f(t)=l_1\phi_3''(t)+l_2\phi_4''(t).$$
	 	Then we can partition the range of $f$ into sets $R_s$ with $0\le s\le j$, each of which is the union of at most two intervals of length $\sim 2^s$, such that for each $v\in R_s$ we have
	 	$$|f^{-1}(v+[-O(1),O(1)])\cap[\frac12,1]|\lesssim \frac1{\sqrt{2^{j+s}}}.$$
	 	All implicit constants are universal over all pairs of such $\phi_3$, $\phi_4$ and over $l_1,l_2,s$.
	 \end{lem}
	 \begin{proof}
	 	The result is trivial if $l_1=l_2=0$, so we will next assume that $\max\{|l_1|,|l_2|\}\ge 1$.
	 	
	 	We restrict $f$ to the interval $[\frac12,1]$.	
	 	Since $$ \begin{bmatrix}f'(t)\\f''(t)\end{bmatrix}=\begin{bmatrix}
	 	\phi_3^{(3)}(t)&\phi_4^{(3)}(t)\\ \phi_3^{(4)}(t)&\phi_4^{(4)}(t)
	 	\end{bmatrix}\begin{bmatrix}l_1\\l_2\end{bmatrix},$$
	 	\eqref{2} implies that for each $t\in[\frac12,1]$ we have
	 	\begin{equation}
	 	\label{40}
	 	\max\{|f'(t)|,|f''(t)|\}\sim 2^j.
	 	\end{equation}
	 	
	 	We let $t_0$ be a point in $[\frac12,1]$ where $|f'|$ attains its minimum. If $|f'(t_0)|\sim 2^j$, then we may take $R_j$ to be the whole range of $f$, and all other $R_s$ to be empty. Indeed, the Mean Value Theorem shows that
	 	$$|f(t_1)-f(t_2)|\gtrsim 1$$
	 	whenever $|t_1-t_2|\gtrsim 2^{-j}$. It is worth observing that if $|l_1|\gg |l_2|$, then \eqref{3} would immediately guarantee that $|f'(t_0)|\sim 2^j$.
	 	
	 	We now assume that $|f'(t_0)|\ll 2^{j}$. Due to \eqref{40}, we must have that $|f''(t_0)|\sim 2^j$.  We write for $t\in [\frac12,1]$
	 	\begin{equation}
	 	\label{41}
	 	f(t)=f(t_0)+f'(t_0)(t-t_0)+\frac{f''(t_0)(t-t_0)^2}{2}+O(2^j(t-t_0)^3).
	 	\end{equation}
	 	\\
	 	\\
	 	Case 1. Consider $s$ with $2^j\ge 2^s> C \max\{\frac{|f'(t_0)|^2}{2^j},1\}$, for some large enough $C$ independent of $j$. Using this and \eqref{41}, we see that
	 	\begin{equation}
	 	\label{104}
	 	|f(t)-f(t_0)|\ll 2^s\;\text{ whenever }|t-t_0|\ll 2^{\frac{s-j}{2}}.
	 	\end{equation}

	 	 Define
	 	$$R_s=\{v:\;|v-f(t_0)|\sim 2^s\}.$$
	 	Let $v\in R_s$ and let $w=v+O(1)$. Thus, we also have $|w-f(t_0)|\sim 2^s$. Let $t_1,t_2$ be such that $f(t_1)=v$, $f(t_2)=w$.
	 	Using \eqref{104} it follows that $|t_1-t_0|,|t_2-t_0|\gtrsim 2^{\frac{s-j}{2}}$. Our assumption shows that $2^{\frac{s-j}{2}}\gg \frac{|f'(t_0)|}{2^j}$. Thus, $|t_1-t_0|,|t_2-t_0|\gg \frac{|f'(t_0)|}{2^j}$, and using \eqref{41} again we conclude that
	 	$$|f(t_i)-f(t_0)|\sim 2^j|t_i-t_0|^2.$$
	 	Thus, $|t_i-t_0|\sim 2^{\frac{s-j}2}$. Using again  \eqref{41} we find that if $t_1,t_2$ are on the same side of $t_0$ then
	 	$$|f(t_1)-f(t_2)|\sim 2^{\frac{s+j}{2}}|t_1-t_2|.$$
	 	We conclude that $|t_1-t_2|\lesssim \frac1{\sqrt{2^{j+s}}}$, as desired.
	 	\\
	 	\\
	 	Next, we define $R_s$ for smaller values of $s$. We distinguish two cases.
	 	\\
	 	\\
	 	Case 2a.  Assume now that $|f'(t_0)|\le 2^{j/2}$. For  $s$ such that  $2^s$ is the largest dyadic power $\le C \max\{\frac{|f'(t_0)|^2}{2^j},1\}=C$ we define
	 	$$R_s=\{v:\;|v-f(t_0)|\lesssim 2^s\}.$$
	 	We also let $R_{s'}=\emptyset$ for smaller values of $s'$. Let $v\in R_s$ and $w=v+O(1)$. Let $t_1,t_2$ be such that $f(t_1)=v$, $f(t_2)=w$. Since in fact  $|f(t_i)-f(t_0)|\lesssim 1$, \eqref{41} forces $|t_i-t_0|\lesssim 2^{-j/2}\sim \frac{1}{\sqrt{2^{j+s}}}$, as desired.
	 	\\
	 	\\
	 	Case 2b.  Assume now that $|f'(t_0)|> 2^{j/2}$. For $s$ such that  $2^s$ is the largest dyadic power $\le C \max\{\frac{|f'(t_0)|^2}{2^j},1\}=C\frac{|f'(t_0)|^2}{2^j}$
	 	we define
	 	$$R_s=\{v:\;|v-f(t_0)|\lesssim 2^s\}.$$
	 	We also let $R_{s'}=\emptyset$ for smaller values of $s'$.
	 	Let $v\in R_s$ and $w=v+O(1)$. Let $t_1,t_2$ be such that $f(t_1)=v$, $f(t_2)=w$. Using  that $|f'(t)|\ge |f'(t_0)|$ for all $t$, we find that
	 	$$|f(t_1)-f(t_2)|\ge |t_1-t_2||f'(t_0)|.$$
	 	We conclude that $$|t_1-t_2|\lesssim \frac{1}{|f'(t_0)|}\sim \frac1{\sqrt{2^{j+s}}},$$
	 	as desired.
	 	
	 \end{proof}

\bigskip

From now on, we will implicitly assume that all Weyl sums are smooth, as in Lemma \ref{22}. This can be easily arranged using partitions of unity, namely  working with smooth $\gamma$ satisfying
$$\sum_{l\in\Z}\gamma(\cdot+l)=1_\R.$$
To simplify notation, these weights will be ignored.
\smallskip

Cover $\Omega$ with unit cubes $B=B_{p,l_1,l_2}=[0,1]\times [p,p+1]\times [l_1,l_1+1]\times [l_2,l_2+1]$ with
$$p\le N,\;\;l_1,l_2\le N^{1/2}.$$
We first write
$$\int_\Omega|\Ec_{I_1}\Ec_{I_2}|^6\sim \sum_{B\subset \Omega}\int_B|\Ec_{I_1}\Ec_{I_2}|^6.$$
We use $l^2$ decoupling (Theorem \ref{49}) on each $B$
$$\int_B|\Ec_{I_1}\Ec_{I_2}|^6\lesssim_\epsilon N^\epsilon (\sum_{J_1\subset I_1}\sum_{J_2\subset I_2}(\int_B|\Ec_{J_1}\Ec_{J_2}|^6)^{1/3})^3$$
where $J_i$ is of the form $[h_i,h_i+N^{1/2}]$. When $x\in B$ and $ J=[h,h+N^{1/2}]$
$$|\Ec_J(x)|=|\sum_{m\le N^{1/2}}e(mu+m^2w+{m^3}v+O(N^{-1/4}))|$$
where
\begin{equation}
\label{huiyfyrfyrfy}
\begin{cases}u=x_1+\frac{2h}{N}x_2+\phi_3'(\frac{h}{N})x_3+\phi_4'(\frac{h}{N})x_4\\w=\frac{x_2}{N}+\frac12\phi_3''(\frac{h}{N})\frac{x_3}{N}+\frac12\phi_4''(\frac{h}{N})\frac{x_4}{N}\\v=\frac{\phi_3'''(\frac{h}{N}){x_3}+\phi_4'''(\frac{h}{N}){x_4}}{6N^2}\end{cases}.
\end{equation}
The term $O(N^{-1/4})$ can be dismissed as it produces tiny errors consistent with square root cancellation. Note that since $v=O(N^{-3/2})$, we have
$$|\sum_{m\le N^{1/2}}e(mu+m^2w+{m^3}v)|\approx |\sum_{m\le N^{1/2}}e(mu+m^2w)|.$$
See Lemma \ref{22} for a rigorous argument. The key point is that we may dismiss the cubic terms.
\smallskip

Write
$$I(h_1,h_2,B)=$$$$\int_{(u_1,u_2,w_1,w_2)\in [0,1]^2\times [\frac{a_1-O(1)}{N},\frac{a_1+O(1)}{N}]\times [\frac{a_2-O(1)}{N},\frac{a_2+O(1)}{N}]}|\prod_{i=1}^2\sum_{m_i\le N^{1/2}}e(m_iu_i+m_i^2w_i)|^6du_1du_2dw_1dw_2,$$
where
\begin{equation}
\label{19}
\begin{cases}a_1=p+\frac{l_1}2\phi_3''(\frac{h_1}{N})+\frac{l_2}2\phi_4''(\frac{h_1}{N})\\
a_2=p+\frac{l_1}2\phi_3''(\frac{h_2}{N})+\frac{l_2}2\phi_4''(\frac{h_2}{N})
\end{cases}.\end{equation}
Via the change of variables with Jacobian $\sim \frac1{N^2}$ (Lemma \ref{18})
$$\begin{cases}u_1=x_1+\frac{2h_1}{N}x_2+\phi_3'(\frac{h_1}{N})x_3+\phi_4'(\frac{h_1}{N})x_4\\w_1=\frac{x_2}{N}+\frac12\phi_3''(\frac{h_1}{N})\frac{x_3}{N}+\frac12\phi_4''(\frac{h_1}{N})\frac{x_4}{N}\\u_2=x_1+\frac{2h_2}{N}x_2+\phi_3'(\frac{h_2}{N})x_3+\phi_4'(\frac{h_2}{N})x_4\\w_2=\frac{x_2}{N}+\frac12\phi_3''(\frac{h_2}{N})\frac{x_3}{N}+\frac12\phi_4''(\frac{h_2}{N})\frac{x_4}{N}\end{cases}$$
we see that
$$\int_B|\Ec_{J_1}\Ec_{J_2}|^6\lesssim N^2I(h_1,h_2,B).$$
Writing $$I_{a}=\int_{[0,1]\times [\frac{a-O(1)}{N},\frac{a+O(1)}{N}]}|\sum_{m\le N^{1/2}}e(mu+m^2w)|^6dudw$$ we find that
$$\int_B|\Ec_{J_1}\Ec_{J_2}|^6\lesssim N^2I_{a_1}I_{a_2}.$$
Let us analyze \eqref{19}. The question is, for fixed $B$, what are the values of $a_1,a_2$ that arise (modulo $O(1)$ error terms), and what is their multiplicity, when $h_1,h_2$ range through the multiples of $N^{1/2}$ in $[1,N]$.

Assume $l_1\sim 2^{j_1}$, $l_2\sim 2^{j_2}$, with $2^{j_1},2^{j_2}\le N^{1/2}$. We may assume $j_1\le j_2$, the other case is completely similar. We apply Lemma \ref{42} to $f(t)=\frac12(l_1\phi_3''(t)+l_2\phi_4''(t))$.
For each $0\le s_1,s_2\le j_2$ and each $p$ we have $O(2^{s_1+s_2})$ pairs $(a_1,a_2)$ of integers with $a_1-p\in R_{s_1}(l_1,l_2)$ and $a_2-p\in R_{s_2}(l_1,l_2)$. Note that we index the intervals $R_{s_i}$ from Lemma \ref{42}  by $l_1,l_2$.
 For each such pair $(a_1,a_2)$, \eqref{19} has $O(\frac{N}{2^{j_2}2^{\frac{s_1+s_2}2}})$ solutions $(h_1,h_2)$.
 When we count solutions, we tolerate error terms of size $O(1)$.
 \medskip

 Thus
\begin{align*}
&\sum_{B\subset \Omega}\int_B|\Ec_{I_1}\Ec_{I_2}|^6\\&\lesssim N^2\sum_{p\le N}\sum_{2^{j_2}\lesssim N^{1/2}}\sum_{2^{j_1}\lesssim 2^{j_2}}\sum_{l_1\sim 2^{j_1}}\sum_{l_2\sim 2^{j_2}}\sum_{s_1,s_2\le j_2}(\frac{N}{2^{j_2+\frac{s_1+s_2}{2}}})^3(\sum_{a_1\in p+ R_{s_1}(l_1,l_2)}\sum_{a_2\in p+ R_{s_2}(l_1,l_2)}I_{a_1}^{1/3}I_{a_2}^{1/3})^3
\\&\lesssim
N^2\sum_{p\le N}\sum_{2^{j_2}\lesssim N^{1/2}}\sum_{2^{j_1}\lesssim 2^{j_2}}2^{j_1+j_2}(\frac{N}{2^{j_2}})^3\sum_{s_1,s_2\le j_2}(\sum_{a_1\in p+ R_{s_1}(l_1,l_2)}I_{a_1}^{2/3})^{3/2}(\sum_{a_2\in p+ R_{s_2}(l_1,l_2)}I_{a_2}^{2/3})^{3/2}.
\end{align*}
The last inequality follows from Cauchy--Schwarz. Next, we observe that $p+R_{s_i}(l_1,l_2)\subset [p-O(2^{j_2}),p+O(2^{j_2})]$. These intervals are roughly the same for roughly $2^{j_2}$ values of $p$. We can thus dominate the above by
\begin{align}
\label{44}
\begin{split}
&\les N^2\sum_{2^{j_2}\lesssim N^{1/2}}\sum_{2^{j_1}\lesssim 2^{j_2}}2^{j_1+j_2}(\frac{N}{2^{j_2}})^32^{j_2}\sum_{H\subset [0,N]\atop{|H|={2^{j_2}}}}(\sum_{a\in H}I_{a}^{2/3})^3\\&\sim  N^5\sum_{2^{j}\lesssim N^{1/2}}\sum_{H\subset [0,N]\atop{|H|={2^{j}}}}(\sum_{a\in H}I_{a}^{2/3})^3.
\end{split}
\end{align}

The sum runs over pairwise disjoint intervals $H$. It is easily seen to be $O(N^8)$, by using  the following lemma with $M=N^{1/2}$.
\begin{lem}
\label{76}	
	Let $$I_{a}=\int_{[0,1]\times [\frac{a-O(1)}{M^2},\frac{a+O(1)}{M^2}]}|\sum_{m\le M}e(mu+m^2w)|^6dudw$$
For each $2^{j}\le M^2$ we have
$$\sum_{H\subset [0,M^2]\atop{|H|={2^{j}}}}(\sum_{a\in H}I_{a}^{2/3})^3\les M^42^{2j}+M^6.$$
\end{lem}
\begin{proof}

The arcs $\{x\in[0,1):\;\dist(x-\frac{b}{q},\Z)\le \frac{1}{qM}\}$, with $1\le b\le q\le M$ and $(b,q)=1$, cover $[0,1)$. They may overlap, which leads to double counting in our argument, but this will be harmless.

We consider the contribution from those $I_a$ with $\frac{a}{M^2}$ in some arc with $q\sim Q$. Here $Q$ is dyadic and $Q\lesssim M$.
We separate the proof into two cases. Note that $H/M^2\subset [0,1]$ and has length $2^j/M^2$. Also, $|b/q-b'/q'|\ge 1/qq'$.
\\
\\
Case 1. $Q^2>\frac{M^2}{2^j}$. Each $H/M^2$ intersects $\lesssim\frac{2^jQ^2}{M^2}$  arcs with $q\sim Q$. For each such $b/q$, and each $1\le 2^k\le \frac{M}{q}$ there are $\sim\frac{M}{2^kQ}$ values of $a$ with
$$|\frac{a}{M^2}-\frac{b}{q}|\sim \frac1{qM2^k}.$$
Call $\A(Q,k)$ the collection of all these $a$.
For each $a\in\A(Q,k)$, Lemma \ref{22} gives
$$I_a\lesssim \frac1{2^kM^2}(M^{1/2}2^{k/2})^6=2^{2k}M.$$

The contribution from  $a\in\A(Q,k)$ is
$$\sum_{H\subset [0,M^2]\atop{|H|={2^{j}}}}(\sum_{a\in H\cap \A(Q,k)}I_{a}^{2/3})^3\lesssim \frac{M^2}{2^j}(\frac{2^jQ^2}{M^2}\frac{M}{2^kQ})^3(M2^{2k})^2=M2^k2^{2j}Q^3.$$
This is easily seen to be $O(2^{2j}M^4)$, since $2^{k}=O(MQ^{-1})$ and  $Q=O(M)$.
The contribution to the full sum  is acceptable, since there are $\les 1$ values of $Q$ and $k$.
\\
\\
Case 2. $Q^2<\frac{M^2}{2^j}$. There are $\lesssim Q^2$   arcs with $q\sim Q$. Essentially, each $H$ is either disjoint from all these (so not contributing at this stage) or (essentially) contained inside one of them. We distinguish two subcases.
\\
\\
(a) If $\frac{2^j}{M^2}<\frac{1}{QM2^k}$ (this is stronger than $Q^2<\frac{M^2}{2^j}$), there are $\frac{1}{QM2^k}\frac{M^2}{2^j}$ intervals $H/M^2$ contained in $[\frac{b}{q}-\frac{1}{QM2^k},\frac{b}{q}+\frac{1}{QM2^k}]$.
Their contribution is
\begin{align*}
\sum_{b<q\sim Q}\sum_{H/M^2\subset [\frac{b}{q}-\frac{1}{QM2^k},\frac{b}{q}+\frac{1}{QM2^k}]\atop{|H|={2^{j}}}}(\sum_{a\in H}I_{a}^{2/3})^3&\lesssim  Q^2\frac{M}{Q2^k2^j}2^{3j}(2^{2k}M)^2\\&=2^{2j}M^3Q2^{3k}.
\end{align*}
Using our assumption, this is $O(2^{-j}M^6)$.
\\
\\
(b) If $\frac{2^j}{M^2}>\frac{1}{QM2^k}$, for each $b/q$ with $q\sim Q$ there is only one $H/M^2$ that intersects
$$|t-\frac{b}{q}|\sim \frac1{qM2^k}$$
with at most $M^2\frac1{qM2^k}$ values of $a$ contributing from $H$. The contribution from the $O(Q^2)$  arcs with denominator $\sim Q$ is
$$\lesssim Q^2(\frac{M}{Q2^k})^3(2^{2k}M)^2=\frac{2^kM^{5}}{Q}.$$
Since $2^k\lesssim M$, this term is $O(M^6)$.

\end{proof}

\begin{re}
\label{69}

One may wonder whether there is a clever way to estimate the sum 	
$$ \sum_{B\subset \Omega}(\sum_{J_1\subset I_1}\sum_{J_2\subset I_2}(\int_B|\Ec_{J_1}\Ec_{J_2}|^6)^{1/3})^3,$$
without using number theory. To this end, the most natural thing to try is to use Minkowski's inequality and to bound this expression by
\begin{equation}
\label{50}
(\sum_{J_1\subset I_1}\sum_{J_2\subset I_2}(\int_\Omega|\Ec_{J_1}\Ec_{J_2}|^6)^{1/3})^3.
\end{equation}
However, a change of variables as before shows that for each $J_1,J_2$
\begin{align*}
\int_\Omega|\Ec_{J_1}\Ec_{J_2}|^6&\ge N^{-1/2}\int_{[0,N^{1/2}]^4}|\Ec_{J_1}\Ec_{J_2}|^6\\&\sim N^{-1/2}N^2[\int_{(u,w)\in[0,N^{1/2}]\times[0,N^{-1/2}]}|\sum_{m=1}^{N^{1/2}}e(mu+m^2w)|^6dudw]^2\\&\sim N^{11/2}.
\end{align*}
Using this approach, the upper bound we get for \eqref{50} is $N^{8+\frac12}$. As in our earlier attempt to use $l^6$ rather than $l^2$ decoupling, this estimate falls short by a factor of $N^{1/2}$ from the sharp upper bound $N^8$.

Also, due to \eqref{72}, the expression \eqref{50} is only getting larger if $J_i$ are replaced with smaller intervals. Thus, decoupling on cubes larger than $B$ (such as $N^{1/2}$-cubes) only worsens our upper bound.

A similar computation shows that the only case of Conjecture \ref{51} that can be approached with $l^6L^6$ decoupling is the case $\alpha=2$ discussed in the previous section.
\end{re}
\bigskip

\section{The case $\frac32<\alpha\le \frac95$}
\label{s7}

Let $\Omega=[0,N^{\frac{2\alpha}{3}-1}]\times [0,N]\times [N^{\frac12}, N^{\alpha-1}]\times [0,N^{\beta-1}]$. Using 1-periodicity in $x_1$, we need to prove that
$$\int_\Omega|\Ec_{I_1}\Ec_{I_2}|^6\lesssim_\epsilon N^{7+\frac{2\alpha}{3}+\epsilon}.$$
We cover $\Omega$ with cubes $B$ of side length $N^{\frac{2\alpha}{3}-1}$ and write
$$\int_\Omega|\Ec_{I_1}\Ec_{I_2}|^6\sim \sum_{B\subset \Omega}\int_B|\Ec_{I_1}\Ec_{I_2}|^6.$$
The size of these cubes is the smallest that will make cubic terms negligible after the decoupling. Since we need $\frac{2\alpha}{3}-1\le \beta-1$ in order to not exceed $\Omega$, this leads to the restriction $\alpha\le \frac{9}{5}$. Note also that the $x_3$ coordinate is $\ge N^{1/2}$. We can afford this omission because of the $\alpha=\frac32$ case discussed in the previous section. Since we are about to decouple on cubes $B$ with size larger than 1, Remark \ref{69} tells us that applying this method for $x$ near the origin leads to losses. Our next argument will make explicit use of the fact that $x_3$ is away from the origin.

We use $l^2$ decoupling (Theorem \ref{49}) on each $B$ (or rather $NB$, after rescaling)
$$\int_B|\Ec_{I_1}\Ec_{I_2}|^6\lesssim_\epsilon N^\epsilon (\sum_{J_1\subset I_1}\sum_{J_2\subset I_2}(\int_B|\Ec_{J_1}\Ec_{J_2}|^6)^{1/3})^3$$
where $J_i$ is of the form $[h_i,h_i+M]$, with $M=N^{1-\frac{\alpha}{3}}$. We write $B$ as
$$[0,N^{\frac{2\alpha}{3}-1}]\times [pN^{\frac{2\alpha}{3}-1},(p+1)N^{\frac{2\alpha}{3}-1}]\times [l_1N^{\frac{2\alpha}{3}-1},(l_1+1)N^{\frac{2\alpha}{3}-1}]\times[l_2N^{\frac{2\alpha}{3}-1},(l_2+1)N^{\frac{2\alpha}{3}-1}]$$
with the integers $0\le p\le M^2$, $N^{\frac32-\frac{2\alpha}{3}}\le l_1\le N^{\frac{\alpha}{3}}$ and $0\le l_2\le N^{3-\frac{5\alpha}{3}}$. Since $\alpha>\frac32$, we have that $l_1\gg l_2$.

We let as before  $$I_{a}=\int_{[0,1]\times [\frac{a-O(1)}{M^2},\frac{a+O(1)}{M^2}]}|\sum_{m\le M}e(mu+m^2w)|^6dudw$$

With a change of variables as in the previous section, we have
$$\int_B|\Ec_{J_1}\Ec_{J_2}|^6\sim N^2(N^{\frac{2\alpha}{3}-1})^2I_{a_1}I_{a_2}$$
where
\begin{equation}
\label{70}
\begin{cases}a_1=p+\frac{l_1}2\phi_3''(\frac{h_1}{N})+\frac{l_2}2\phi_4''(\frac{h_1}{N})\\
a_2=p+\frac{l_1}2\phi_3''(\frac{h_2}{N})+\frac{l_2}2\phi_4''(\frac{h_2}{N})
\end{cases}.\end{equation}
It is crucial that the cubic (and also the higher order) term is $O(1)$, cf. \eqref{huiyfyrfyrfy}
$$m^3\frac{\phi_3'''(\frac{h}{N}){x_3}+\phi_4'''(\frac{h}{N}){x_4}}{N^2}=O(1), \;\;\forall x\in\Omega,$$
so it may be neglected according to Lemma \ref{22}.
\medskip

If $l_1\sim 2^{j_1}$, it is immediate that $|a_1-p|\lesssim 2^{j_1}$, $|a_2-p|\lesssim 2^{j_1}$. Also, for fixed $a_1,a_2,p,l_1,l_2$, \eqref{70} has $O((\frac{N}{M2^{j_1}})^2)$ solutions $(h_1,h_2)$, modulo $O(1)$. We do not need Lemma \ref{42} here, since this time $l_1$ is much larger than $l_2$. We now dominate $\int_\Omega|\Ec_{I_1}\Ec_{I_2}|^6$ as before by
$$\lesssim_\epsilon N^{\epsilon}\sum_{j_1:\;N^{\frac32-\frac{2\alpha}{3}}\le 2^{j_1}\le N^{\frac{\alpha}{3}}}\sum_{l_1\sim 2^{j_1}}\sum_{j_2:\;2^{j_2}\le N^{3-\frac{5\alpha}{3}}}\sum_{l_2\sim 2^{j_2}}(\frac{N}{M2^{j_1}})^6N^2(N^{\frac{2\alpha}{3}-1})^2\sum_{p\le M^2}(\sum_{a:\;|a-p|\lesssim 2^{j_1}}I_{a}^{1/3})^6.$$
We use Cauchy--Schwarz for the last expression to dominate the above by
$$N^{\epsilon}\sum_{j:\;N^{\frac32-\frac{2\alpha}{3}}\le 2^{j}\le N^{\frac{\alpha}{3}}}2^{j}N^{3-\frac{5\alpha}{3}}2^{-6j}N^{2\alpha}N^2(N^{\frac{2\alpha}{3}-1})^22^{j}2^{3j}\sum_{|H|\sim 2^{j}}(\sum_{a\in H}I_{a}^{2/3})^3$$$$=N^{3+\frac{5\alpha}{3}+\epsilon}\sum_{j:\;N^{\frac32-\frac{2\alpha}{3}}\le 2^{j}\le N^{\frac{\alpha}{3}}}2^{-j}\sum_{|H|\sim 2^{j}}(\sum_{a\in H}I_{a}^{2/3})^3.$$
Using Lemma \ref{76}, this is dominated by
$$N^{3+\frac{5\alpha}{3}+\epsilon}\sum_{j:\;N^{\frac32-\frac{2\alpha}{3}}\le 2^{j}\le N^{\frac{\alpha}{3}}}(M^42^{j}+M^62^{-j})\lesssim N^{\epsilon}(N^{7+\frac{2\alpha}{3}}+N^{\frac{15}{2}+\frac{\alpha}{3}}).$$
This is $O(N^{7+\frac{2\alpha}{3}+\epsilon})$, as desired, since $\alpha>\frac{3}{2}$.
\bigskip

\section{The case $\frac95\le \alpha<2$}
\label{s8}
Let $$\Omega=[0,N^\delta]\times [0,N]\times [N^{\frac45},N^{\alpha-1}]\times [0,N^{\beta-1}].$$
Because of the case addressed in the previous section, we may assume $x_3\ge N^{\frac45}$. This will buy us  some extra flexibility in choosing $\delta$. In fact, we can work with  any $\delta$ satisfying
\begin{equation}
\label{83}
2-\frac{3\beta}{2}\le  \delta\le \frac95-\beta.\end{equation}
We need to prove that
$$\int_{\Omega}|\Ec_{I_1}\Ec_{I_2}|^6\lesssim_\epsilon N^{8+\delta+\epsilon}.$$
We will first decouple on cubes $B$ with side length $N^{\beta-1}$. This is the largest size that is available to us, due to the range in the $x_4$ variable.  Unlike the case from the previous section, the resulting intervals are not small enough to make the cubic terms negligible, to allow us to use estimates for quadratic Weyl sums. We will accomplish that by means of a further decoupling, on cubes of side length $N^\delta$, similar to the case $\alpha=2$ described earlier.

\smallskip

To get started, we use $l^2$ decoupling (Theorem \ref{49}) on each cube $B$ of side length $N^{\beta-1}$
$$\int_B|\Ec_{I_1}\Ec_{I_2}|^6\lesssim_\epsilon N^\epsilon (\sum_{J_1\subset I_1}\sum_{J_2\subset I_2}(\int_B|\Ec_{J_1}\Ec_{J_2}|^6)^{1/3})^3,$$
where $J_i=[h_i,h_i+M]$ has length  $M=N^{1-\frac\beta2}$.

Next, we cover $\Omega$ with boxes
$$\Delta=[0,N^\delta]\times [pN^\delta,(p+1)N^{\delta}]\times [lN^{\delta},(l+1)N^\delta]\times [0,N^{\beta-1}],$$
with $p\le N^{1-\delta}$, $N^{\frac45-\delta}\le l\le N^{\alpha-1-\delta}$. If we sum up the above inequality over cubes  $B\subset \Delta$ and use Minkowski's inequality, we find
\begin{equation}
\label{86}
\int_\Delta|\Ec_{I_1}\Ec_{I_2}|^6\lesssim_\epsilon N^\epsilon (\sum_{J_1\subset I_1}\sum_{J_2\subset I_2}(\sum_{B\subset \Delta}\int_B|\Ec_{J_1}\Ec_{J_2}|^6)^{1/3})^3.
\end{equation}
Next, we fix $J_1,J_2$ and perform a second decoupling for the term $\int_\Delta|\Ec_{J_1}\Ec_{J_2}|^6$. We proceed as in Section \ref{Bsarg}

\begin{align*}
\int_{B}|\Ec_{J_1}\Ec_{J_2}|^6&=N^{-4}\int_{NB}|\Ec_{J_1}(\frac{\cdot}{N})\Ec_{J_2}(\frac{\cdot}{N})|^6\\&\lesssim N^{-4-4\beta}\int_{NB}|\Ec_{J_1}(\frac{\cdot}{N})|^6\int_{NB}|\Ec_{J_2}(\frac{\cdot}{N})|^6=N^{4-4\beta}\int_{B}|\Ec_{J_1}|^6\int_{B}|\Ec_{J_2}|^6.
\end{align*}
Then
$$\sum_{B\subset \Delta}\int_{B}|\Ec_{J_1}|^6\int_{B}|\Ec_{J_2}|^6\lesssim N^{4-4\beta}\int_{\Delta}dx\int_{(y,z)\in [0,N^{\beta-1}]^4\times [0,N^{\beta-1}]^4}|\Ec_{J_1}(x+y)\Ec_{J_2}(x+z)|^6dydz.$$
Combining these two and using periodicity in the $y_1,z_1$ variables we get
$$\int_{\Delta}|\Ec_{J_1}\Ec_{J_2}|^6\lesssim N^{6-6\beta-2\delta}\times$$
$$\int_{(x_1,x_4,y_2,y_3,y_4,z_2,z_3,z_4)\in S}dx_1\ldots dz_4\int_{y_1,z_1\in[0,N^\delta]\atop{x_2\in[pN^\delta,(p+1)N^\delta]\atop{x_3\in[lN^\delta,(l+1)N^\delta]}}}|\Ec_{J_1}(x+y)\Ec_{J_2}(x+z)|^6dy_1dz_1dx_2dx_3,$$
where $S$ is characterized by
$$0\le x_1\in [0, N^\delta],\;x_4,y_2,y_3,y_4,z_2,z_3,z_4\in [0,N^{\beta-1}].$$
We seek to estimate the second integral uniformly over $x_1,x_4,y_2,y_3,y_4,z_2,z_3,z_4$. With these variables fixed,
we make the affine change of variables $(y_1,z_1,x_2,x_3)\mapsto (u_1,u_2,w_1,w_2)$
$$
\begin{cases}u_1=(y_1+x_1)+\frac{2h_1}{N}(x_2+y_2)+\phi_3'(\frac{h_1}{N})(x_3+y_3)+\phi_4'(\frac{h_1}{N})(x_4+y_4)\\u_2=(z_1+x_1)+\frac{2h_2}{N}(x_2+z_2)+\phi_3'(\frac{h_2}{N})(x_3+z_3)+\phi_4'(\frac{h_2}{N})(x_4+y_4) \\w_1=\frac{x_2}{N}+\phi_3''(\frac{h_1}{N})\frac{x_3}{2N}\\w_2=\frac{x_2}{N}+\phi_3''(\frac{h_2}{N})\frac{x_3}{2N}\end{cases}.
$$
The  Jacobian is $\sim \frac1{N^2}$, due to \eqref{3}. The second integral is comparable to
\begin{equation}
\label{81}
 N^2\int_{(u_i,w_i)\in[0,N^\delta]\times [\frac{a_i-O(1)}{M_*^2},\frac{a_i+O(1)}{M_*^2}]}\prod_{i=1}^2|\sum_{m_i\le M}e(m_iu_i+m_i^2w_i+\eta_i(m_i)x_3)|^6du_1dw_1du_2dw_2.
\end{equation}
Here $M_*=N^{\frac12-\frac\delta2}$,
$a_i=p+l\frac{\phi_3''(\frac{h_i}{N})}{2}$ and $\eta_i(m)={m^3}\frac{\phi_3'''(\frac{h_i}{N})}{3!N^2}+{m^4}\frac{\phi_3''''(\frac{h_i}{N})}{4!N^3}+\ldots$. Note that since $\frac{v}N=O(M^{-2})$ for $v$ equal to any of the variables $x_4,y_2,y_3,y_4,z_2,z_3,z_4$, we have dismissed the contribution  of these variables associated with quadratic (as well as the higher order) terms. See Lemma \ref{22}.

We may apply again Theorem \ref{65}, using that $x_3=A(w_1-w_2)$ with $A=O(N)$. Note however that this time we cannot decouple into point masses (as in \eqref{59}), since $M_*$ is significantly larger than 1. Instead, applying \eqref{80} with $N=(\frac{M}{M_*})^2$ we dominate \eqref{81} by
\begin{equation}
\label{101}
N^{2+\epsilon}\times
\end{equation}
$$(\sum_{J_1',J_2'}[\int_{(u_i,w_i)\in [0,N^\delta]\times [\frac{a_i-O(1)}{M_*^2},\frac{a_i+O(1)}{M_*^2}]}\prod_{i=1}^2|\sum_{m_i\in J_i'}e(m_iu_i+m_i^2w_i+\eta_i(m_i)x_3)|^6du_1dw_1du_2dw_2]^\frac13)^3.
$$
The intervals $J_i'$ partitioning $[1,M]$ have length $M_*$. What we have gained by doing this decoupling is that, when $m_i$ is confined to a small interval $J_i'=[h_i',h_i'+M_*]$, the contribution of the term
\begin{equation}
\label{yqp[sq[ps]]}
\eta_i(m_i)=\eta_i(h_i'+m_i')=\eta_i(h_i')+\eta_i'(h_i')m_i'+\eta_i''(h_i')\frac{(m_i')^2}{2}+O(\frac{(m_i')^3}{N^2})
\end{equation}
 can be neglected. To see this, note first that
$$\eta_i''(h_i')=\sum_{n\ge 3}\phi_3^{(n)}(\frac{h_i}{N})\frac{(h_i')^{n-2}}{N^{n-1}(n-2)!}.$$
Making another linear change of variables such that
$$w_i'=w_i+\frac{\eta_i''(h_i')}{2}A(w_1-w_2),$$
we write, using that $|x_3|\le N^{\alpha-1}$
$$\prod_{i=1}^2|\sum_{m_i\in J_i'}e(m_iu_i+m_i^2w_i+\eta_i(m_i)x_3)|=\prod_{i=1}^2|\sum_{m_i'\in [1,M_*]}e(m_i'u_i'+(m_i')^2w_i'+O((m_i')^3\frac{N^{\alpha-1}}{N^2}))|.$$
The range of $w_i'$ is (a subset of) $[\frac{a_i'-O(1)}{M_*^2},\frac{a_i'+O(1)}{M_*^2}]$, where
$$a_i'=p+\frac{l}{2}\sum_{n\ge 2}\phi_3^{(n)}(\frac{h_i}{N})\frac{(h_i')^{n-2}}{N^{n-2}(n-2)!}=p+l\frac{\phi_3^{''}(\frac{h_i+h_i'}{N})}{2}.$$
Since $\alpha-1-\frac{\beta}{2}\le  \delta$ by \eqref{83}, we have that $a_i-a_i'=O(1)$.
Thus, the  quadratic term in \eqref{yqp[sq[ps]]} will not affect the domain of integration. Moreover, the contribution of the higher order terms in \eqref{yqp[sq[ps]]} is negligible (cf. Lemma \ref{22}), as long as we can guarantee that   we have
$\frac{N^{\alpha-1}}{N^2}=O(M_*^{-3})$. This is equivalent to
$\delta\ge 1-\frac{2\beta}3$, and follows from \eqref{83} and the fact that $\beta\le \frac{6}{5}$.
Under this assumption, we dominate \eqref{101} by
$$N^{2+\epsilon}\left(\sum_{J_1',J_2'}\left[\prod_{i=1}^2\int_{[0,N^\delta]\times [\frac{a_i-O(1)}{M_*^2},\frac{a_i+O(1)}{M_*^2}]}|\sum_{m_i\in J_i'}e(m_iu_i+m_i^2w_i)|^6du_idw_i\right]^\frac13\right)^3.$$
This is $\sim N^{2+2\delta+\epsilon}(\frac{M}{M_*})^6I_{a_1}I_{a_2}$, where
$$I_{a}=\int_{[0,1]\times [\frac{a-O(1)}{M_*^2},\frac{a+O(1)}{M_*^2}]}|\sum_{m\le M_*}e(mu+m^2w)|^6dudw$$
is independent of $J_1',J_2'$. Recall that
\begin{equation}
\label{85}
\begin{cases}a_1=p+l\frac{\phi_3''(\frac{h_1}{N})}{2}\\a_2=p+l\frac{\phi_3''(\frac{h_2}{N})}{2}
\end{cases}.
\end{equation}
Assume now that $l\sim 2^j$, with $$N^{\frac45-\delta}\lesssim 2^j\lesssim N^{\alpha-1-\delta}.$$
For fixed $p,l$, and fixed $(a_1,a_2)$ (within a factor of  $O(1)$), the system \eqref{85} has $\lesssim (\frac{N}{2^jM})^2$ solutions $(h_1,h_2)$. Getting back to \eqref{86}, summing over $\Delta\subset \Omega$ we find that
$$\int_\Omega|\Ec_{I_1}\Ec_{I_2}|^6\lesssim_\epsilon$$$$ N^{6-6\beta-2\delta}|S|N^{2+2\delta+\epsilon} N^{3+3\delta}\sum_{p\le M_*^2}\sum_{N^{\frac45-\delta}\lesssim 2^j\lesssim N^{\alpha-1-\delta}}2^{-6j}\sum_{l\sim 2^j}(\sum_{|a-p|\lesssim 2^j}I_a^{1/3})^6.$$
We use  Cauchy--Schwarz to dominate this by
$$N^{4+4\delta+\beta+\epsilon} \sum_{N^{\frac45-\delta}\lesssim 2^j\lesssim N^{\alpha-1-\delta}}\sum_{H\subset [1,M_*^2]\atop{|H|=2^j}}2^{-j}(\sum_{a\in H}I_a^{2/3})^3.$$
Using Lemma \ref{76}, it remains to check that
$$ \sum_{N^{\frac45-\delta}\lesssim 2^j\lesssim N^{\alpha-1-\delta}}M_*^42^j+      \sum_{N^{\frac45-\delta}\lesssim 2^j\lesssim N^{\alpha-1-\delta}} M_*^62^{-j}\lesssim N^{4-\beta-3\delta}.$$
The first sum is in order, since $\alpha+\beta=3$. So is the second sum, as long as $\delta\le \frac{9}{5}-\beta$, which is guaranteed by \eqref{83}.

\section{Proof of Theorem \ref{4linver}}
\label{oriruigui}

This  section shows that Theorem \ref{4} implies Theorem \ref{4linver}. The argument is inspired by \cite{Bo}.
\medskip

The parameter $K$ will be very large and universal, independent of $N$, $\phi_k$. The larger the $K$ we choose to work with, the smaller the $\epsilon$ from the $N^\epsilon$ loss will be at the end of the section.

\begin{pr}
	\label{9}
Assume $\alpha+\beta=3$ and $\frac32\le \alpha\le 2.$
Assume $\phi_3,\phi_4: (0,3)\to\R$ are real analytic and satisfy \eqref{1}, \eqref{2} and \eqref{3}. Let as before $\omega_3=[0,N^\alpha]$, $\omega_4=[0,N^\beta]$ and
	$$\Ec_{I,N}(x)=\sum_{n\in I}e(nx_1+n^2x_2+\phi_3(\frac{n}{N})x_3+\phi_4(\frac{n}{N})x_4).$$
	
	We consider arbitrary integers $N_0,M$ satisfying  $1\le M\le \frac{N_0}{K}$ and $N_0+[M,2M]\subset [\frac{N}{2},N]$.
	
	Let $H_1,H_2$ be intervals of length $\frac{M}K$ inside $N_0+[M,2M]$ such that $\dist(H_1,H_2)\ge \frac{M}{K}$. Then
	$$\int_{[0,1]\times [0,1]\times \omega_3\times \omega_4}|\Ec_{H_1,N}(x)\Ec_{H_2,N}(x)|^6\lesssim_\epsilon N^{7}M^{2+\epsilon}.$$	
\end{pr}
\begin{proof}
	Write $H_1=N_0+I_1$, $H_2=N_0+I_2$ with $I_1,I_2$ intervals of length $\frac{M}K$ inside $[M,2M]$ and with separation $\ge \frac{M}{K}$.
	We use the following expansion, certainly valid for all $m$ in $I_i$.
	\begin{align*}
	\phi_3(\frac{N_0+m}{N})&=Q_3(m)+\sum_{n\ge 3}\frac{\phi_3^{(n)}(\frac{N_0}{N})}{n!}(\frac{m}{N})^n\\&=Q_3(m)+(\frac{M}{N})^3\sum_{n\ge 3}\frac{\phi_3^{(n)}(\frac{N_0}{N})(\frac{M}{N})^{n-3}}{n!}(\frac{m}{M})^n.
	\end{align*}
	Here $Q_3(m)=A+Bm+Cm^2$ with $B=O(\frac1N)$, $C=O(\frac1{N^2})$.
	We introduce the analogue $\tilde{\phi_3}$ of $\phi_3$ at scale $M$
	$$\tilde{\phi_3}(t)=\sum_{n\ge 3}\frac{\phi_3^{(n)}(\frac{N_0}{N})(\frac{M}{N})^{n-3}}{n!}t^n.$$
	This series is convergent as long as $\frac{N_0}{N}+t\in (0,3)$, so the new function is certainly real analytic on $(0,2)$, since $N_0\le N$.
	
	Let $\delta>0$ be conveniently small.  By choosing $K$ large enough we can make $\frac{M}{N}$ as small as we wish, so we may guarantee that for each $t\in[\frac12,1]$  we have
	\begin{equation}
	\label{6}
	|\tilde{\phi}_3^{(3)}(t)-{\phi}_3^{(3)}(\frac{N_0}{N})|\le \delta.
	\end{equation}
	Thus, we can guarantee \eqref{3} for  $\tilde{\phi}_3$, with a slightly smaller, but uniform value of $A_4$. The same will  work with \eqref{1} and \eqref{2}, as it will soon become clear. To this end, we may also enforce
	\begin{equation}
	\label{7}
	|\tilde{\phi}_3^{(4)}(t)|\le \delta.
	\end{equation}

	We also define, with $Q_4(m)=D+Em+Fm^2$ satisfying $E=O(\frac1N)$, $F=O(\frac1N^{2})$
	\begin{align*}
	\phi_4(\frac{N_0+m}{N})&=Q_4(m)+\sum_{n\ge 3}\frac{\phi_4^{(n)}(\frac{N_0}{N})}{n!}(\frac{m}{N})^n\\&=Q_4(m)+\frac{\phi_4^{(3)}(\frac{N_0}{N})}{3!}(\frac{m}{N})^3+(\frac{M}{N})^4\sum_{n\ge 4}\frac{\phi_4^{(n)}(\frac{N_0}{N})(\frac{M}{N})^{n-4}}{n!}(\frac{m}{M})^n.\end{align*}
	The last two terms are equal to
	$$
	\frac{\phi_4^{(3)}(\frac{N_0}{N})}{\phi_3^{(3)}(\frac{N_0}{N})}(\frac{M}{N})^3\tilde{\phi}_3(\frac{m}{M})+(\frac{M}{N})^4\sum_{n\ge 4}
	\frac{\phi_4^{(n)}(\frac{N_0}{N})\phi_3^{(3)}(\frac{N_0}{N})-\phi_4^{(3)}(\frac{N_0}{N})\phi_3^{(n)}(\frac{N_0}{N})}{\phi_3^{(3)}(\frac{N_0}{N})n!}(\frac{M}{N})^{n-4}(\frac{m}{M})^n.$$
	Let $\tilde{\phi}_4$ be the analogue of $\phi_4$ at scale $M$ defined by
	$$
	\tilde{\phi}_4(t)=\sum_{n\ge 4}
	\frac{\phi_4^{(n)}(\frac{N_0}{N})\phi_3^{(3)}(\frac{N_0}{N})-\phi_4^{(3)}(\frac{N_0}{N})\phi_3^{(n)}(\frac{N_0}{N})}{\phi_3^{(3)}(\frac{N_0}{N})n!}(\frac{M}{N})^{n-4}t^n.
	$$
	As before, by choosing $K$ large enough, we can arrange that for all $t\in [\frac12,1]$ $$|\tilde{\phi}_4^{(4)}(t)-\frac{\phi_4^{(4)}(\frac{N_0}{N})\phi_3^{(3)}(\frac{N_0}{N})-\phi_4^{(3)}(\frac{N_0}{N})\phi_3^{(4)}(\frac{N_0}{N})}{\phi_3^{(3)}(\frac{N_0}{N})}|\le \delta.$$
	Combining this with \eqref{6} and \eqref{7} we may arrange that
	$$\det
	\begin{bmatrix}
	\tilde{\phi}_3^{(3)}(t)&\tilde{\phi}_3^{(4)}(t)\\ \tilde{\phi}_4^{(3)}(s)&\tilde{\phi}_4^{(4)}(s)\end{bmatrix}-\det
	\begin{bmatrix}
	\phi_3^{(3)}(\frac{N_0}{N})&\phi_3^{(4)}(\frac{N_0}{N})\\ \phi_4^{(3)}(\frac{N_0}{N})&\phi_4^{(4)}(\frac{N_0}{N})\end{bmatrix}$$
	is as small in absolute value as we wish, uniformly over $t,s\in[\frac12,1]$. In particular, we can guarantee \eqref{2} for the pair $(\tilde{\phi}_3,\tilde{\phi}_4)$, with slightly modified, but uniform values of $A_2,A_3$. Similar comments apply regarding \eqref{1}.
	
	Now
	\begin{align*}
	|\Ec_{H_k,N}(x)|&=|\sum_{m\in I_k}e(mx_1+(m^2+2N_0m)x_2+\phi_3(\frac{N_0+m}{N})x_3+\phi_4(\frac{N_0+m}{N})x_4)|\\&=|\sum_{m\in I_i}e(m(x_1+2N_0x_2+Bx_3+Ex_4)+m^2(x_2+Cx_3+Fx_4)\\&+(\frac{M}{N})^3\tilde{\phi}_3(\frac{m}{M})(x_3+\frac{\phi_4^{(3)}(\frac{N_0}{N})}{\phi_3^{(3)}(\frac{N_0}{N})}x_4)\\&+(\frac{M}{N})^4\tilde{\phi}_4(\frac{m}{M})x_4)|.
	\end{align*}
	Recall $N_0\sim N$, $B,E=O(1/N)$, $C,F=O(1/N^2)$.
	We make the change of variables
	$$\begin{cases}
	y_1=x_1+2N_0x_2+Bx_3+Ex_4\\y_2=x_2+Cx_3+Fx_4\\y_3=(\frac{M}{N})^3(x_3+\frac{\phi_4^{(3)}(\frac{N_0}{N})}{\phi_3^{(3)}(\frac{N_0}{N})}x_4)\\y_4=(\frac{M}{N})^4x_4
	\end{cases}.$$
	Due to  periodicity, we may extend the range of $x_1$ to $[0,N_0]$. This linear transformation maps $[0,N_0]\times [0,1]\times \omega_3\times \omega_4$ to a subset of a box $\tilde{\omega}_1\times\tilde{\omega}_2\times\tilde{\omega}_3\times\tilde{\omega_4}$ centered at the origin, with dimensions roughly $N_0,1, M^3N^{\alpha-3}, M^4N^{\beta-4}$.
	
	Thus
	$$|\Ec_{H_k,N}(x)|=|\Ec_{I_k,M}(y)|$$
	where
	$$\Ec_{I_k,M}(y)=\sum_{m\in I_k}e(my_1+m^2y_2+\tilde{\phi}_3(\frac{m}{M})y_3+\tilde{\phi}_4(\frac{m}{M})y_4).$$
	We may write, using again periodicity in $y_1$ and $y_2$
	\begin{align*}
	\int_{[0,1]\times [0,1]\times \omega_3\times \omega_4}|\Ec_{H_1,N}(x)\Ec_{H_2,N}(x)|^6&=\frac1{N_0}\int_{[0,N_0]\times [0,1]\times \omega_3\times \omega_4}|\Ec_{H_1,N}(x)\Ec_{H_2,N}(x)|^6\\&\le(\frac{N}{M})^7\int_{[0,1]\times[0,1]\times\tilde{\omega}_3\times\tilde{\omega_4}}|\Ec_{I_1,M}(y)\Ec_{I_2,M}(y)|^6.
	\end{align*}
	Finally, we use Theorem \ref{4} with $N=M$, noting  that  $\tilde{\omega}_3\subset [-M^\alpha,M^\alpha]$	and $\tilde{\omega}_4\subset [-M^\beta,M^\beta]$, to estimate the last expression by
	$$(\frac{N}{M})^7M^{9+\epsilon}=N^7M^{2+\epsilon}.$$
\end{proof}	
\smallskip

We can now prove  Theorem \ref{4linver}.

	Choose $K$ large enough, depending on $\epsilon$.	
	Write $\Hc_n(I)$ for the collection of dyadic intervals in $I$ with length $\frac{N}{2K^n}$. We write $H_1\not\sim H_2$ to imply that $H_1,H_2$ are not neighbors. Then
	$$|\Ec_{I,N}(x)|\le 3\max_{H\in \Hc_1(I)}|\Ec_{H,N}(x)|+K^{10}\max_{H_1\not\sim H_2\in\Hc_1(I)}|\Ec_{H_1,N}(x)\Ec_{H_2,N}(x)|^{1/2}.$$
	We repeat this inequality until we reach intervals in $\Hc_{l}$ of length $\sim 1$,  that is  $K^l\sim N$.	We have
	\begin{align*}
	|\Ec_{I,N}(x)|&\lesssim 3^l+l3^lK^{10}\max_{1\le n\le l}\max_{H\in\Hc_n(I)}\max_{H_1\not\sim H_2\in\Hc_{n+1}(H)}|\Ec_{H_1,N}(x)\Ec_{H_2,N}(x)|^{1/2}\\&\lesssim (\log N)N^{\log_K3}\max_{1\le n\le l}\max_{H\in\Hc_n(I)}\max_{H_1\not\sim H_2\in\Hc_{n+1}(H)}|\Ec_{H_1,N}(x)\Ec_{H_2,N}(x)|^{1/2}.
	\end{align*}
	Using Corollary \ref{9} we finish the proof
	\begin{align*}
	\int_{[0,1]\times [0,1]\times \omega_3\times \omega_4}& |\Ec_{I,N}|^{12}\\&\lesssim_KN^{\log_K 3}\sum_{n}\sum_{H\in\Hc_n(I)}\max_{H_1\not\sim H_2\in\Hc_{n+1}(H)}\int_{[0,1]\times [0,1]\times \omega_3\times \omega_4}|\Ec_{H_1,N}(x)\Ec_{H_2,N}(x)|^{6}\\&\lesssim_{K,\epsilon}N^{\epsilon+\log_K 3}\sum_n{K^n}N^7(\frac{N}{K^n})^2\\&\lesssim_{K,\epsilon}N^{\epsilon+\log_K 3}N^9.
	\end{align*}
	Choosing $K$ large enough, we may make $\log_K 3$ as small as we wish.

\section{Other values of $p$}
\label{so}
The reason Theorem \ref{4linver} was accessible via the bilinear result in Theorem \ref{4} has to do with the fact that  $6$ is the critical exponent for the decoupling for the parabola (at the canonical scale). Thus, our arguments rely fundamentally on this dimensional reduction. In \cite{DGW}, the small cap decoupling for the parabola is settled, and the associated critical exponents lie between 4 and 6. In principle, this new tool can be used to determine $L^p$ moments for curves in $\R^4$, in the range $8\le p\le 12$.

There are many possible things to consider in this direction. One is the following extension of Conjecture \ref{51},  that we use to illustrate a different type of obstruction that appears in this regime. This was  observed in \cite{Bo34}, in a related context.
\begin{con}[Square root cancellation in $L^{p}$]
	\label{24}
	Let $11\le p\le 12$. Assume $\alpha\ge \beta\ge 0$ satisfy $\alpha+\beta=\frac{p}2-3$.
	Let $\phi_3$, $\phi_4$, $\omega_3$, $\omega_4$ be as in Theorem \ref{4linver}. Then
	$$\int_{[0,1]\times [0,1]\times \omega_3\times \omega_4} |\Ec_{[\frac{N}{2},N],N}|^{p}\lesssim_\epsilon N^{p-3+\epsilon}.$$
\end{con}
The case $\beta=0$ was proved in \cite{DGW} in the larger range $9\le p<12$. As mentioned earlier, when $\beta=0$, the curve collapses to a three dimensional curve. However, the next  result shows that the restriction $p\ge 11$ is needed if $\beta>0.$ The new obstruction can be described as constructive interference on spatially disjoint blocks.

\begin{te}
	Assume $p<11$.	
	Let $\omega_3=[-N^\alpha,N^{\alpha}]$, $\omega_4=[-N^{\beta},N^\beta]$, $\alpha\ge \beta$ and  $\alpha+\beta=\frac{p}2-3$. Assume also that $\beta>0$.

	Then, for some $\delta>0$ and $\phi_3(t)=t^3$, $\phi_4(t)=t^4$ we have
	$$\int_{[-1,1]\times [-1,1]\times \omega_3\times \omega_4} |\Ec_{[\frac{N}{2},N],N}|^{p}\gtrsim N^{p-3+\delta}.$$
\end{te}
\begin{proof}

	Lemma \ref{107} shows that the integral is greater than
	$$\sum_{J\subset I}\int_{[-1,1]\times [-1,1]\times \omega_3\times \omega_4} |\Ec_{J,N}|^{p},$$
	where the sum runs over intervals $J$ of length $M<N$, partitioning $[\frac{N}2,N]$. The parameter $M$ will be determined later.
	In some sense, the components $\Ec_{J,N}$ behave as if they were spatially supported on pairwise disjoint sets.
	
	By periodicity
	$$\int_{[-1,1]\times [-1,1]\times \omega_3\times \omega_4} |\Ec_{J,N}|^{p}=N^{-3}\int_{[-N^2,N^2]\times [-N,N]\times \omega_3\times \omega_4} |\Ec_{J,N}|^{p}.$$
	 Write $H=[h+1,h+M]$. Note that
	$$|\Ec_{J,N}(x)|=|\sum_{1\le m\le M}e(my_1+m^2y_2+m^3y_3+m^4y_4)|$$
	where
	$$\begin{cases}y_1=x_1+2h{x_2}+\frac{3h^2}{N^3}x_3+\frac{4h^3}{N^4}x_4\\y_2=x_2+\frac{3h}{N^3}x_3+\frac{6h^2}{N^4}x_4\\y_3=\frac{x_3}{N^3}+\frac{4h}{N^4}x_4\\y_4=\frac{x_4}{N^4}
	\end{cases}.$$
	This change of variables maps $[-N^2,N^2]\times [-N,N]\times \omega_3\times \omega_4$ to a set containing $$S=[-o(N^2),o(N^2)]\times [-o(N),o(N)]\times [-o(N^{\alpha-3}),o(N^{\alpha-3})]\times [-o(N^{\beta-4}),o(N^{\beta-4})].$$  We have used that $3\ge \alpha\ge \beta-1$. Thus
	$$\int_{[-1,1]\times [-1,1]\times \omega_3\times \omega_4} |\Ec_{J,N}(x)|^{p}dx$$$$\ge N^{4}\int_{S}|\sum_{1\le m\le M}e(my_1+m^2y_2+m^3y_3+m^4y_4)|^{p}dy.$$
	Let $$M=\max\{N^{1-\frac{\alpha}{3}},N^{1-\frac{\beta}4}\}.$$
	Note that since $\beta>0$ we have that $M= N^{1-\epsilon}$ for some  $\epsilon>0$. Note also that
	$$[0,o(M^{-3})]\times [0,o( M^{-4})]\subset[0,o(N^{\alpha-3})]\times [0,o(N^{\beta-4})]. $$
	Using constructive interference we get
	$$\int_{S}|\sum_{1\le m\le M}e(my_1+m^2y_2+m^3y_3+m^4y_4)|^{p}dy\gtrsim M^{p-10}N^3.$$
	Putting things together we conclude that
	$$\int_{[-1,1]\times [-1,1]\times \omega_3\times \omega_4} |\Ec_{[\frac{N}{2},N],N}|^{p}\gtrsim \frac{N}{M}N^4M^{p-10}N^3.$$
	Note that this is $\ge N^{p-3+\delta}$, for some $\delta>0$.
	
\end{proof}
\begin{lem}
\label{107}
Let $\Rc$ be a collections of rectangular boxes $R$ in $\R^n$, with pairwise disjoint doubles $2R$.
Let $F$ be a Schwartz function in $\R^n$ that can be written as
$$F=\sum_{R\in\Rc}F_R,$$
with the spectrum of $F_R$ inside $R$. Then for each $2\le p\le \infty$ we have
$$(\|F_R\|_p)_{l^p(\Rc)}\lesssim \|F\|_p.$$
The implicit constant is independent of $F$ and $\Rc$.
\end{lem}
\begin{proof}
Interpolate between $2$ and $\infty$.
\end{proof}

\section{Auxiliary results}
This section records two auxiliary results that are used repeatedly throughout the paper.

\begin{lem}
	\label{18}
	Assume $t,s\in[\frac12,1]$ satisfy $|t-s|\sim 1$.	
	The Jacobian of the transformation $y=\psi(x)$
	$$\begin{cases}y_1=x_1+2tx_2+\phi_3'(t)x_3+\phi_4'(t)x_4\\y_2=x_1+2sx_2+\phi_3'(s)x_3+\phi_4'(s)x_4\\y_3=\frac{2x_2}{N}+\phi_3''(t)\frac{x_3}{N}+\phi_4''(t)\frac{x_4}{N} \\y_4=\frac{2x_2}{N}+\phi_3''(s)\frac{x_3}{N}+\phi_4''(s)\frac{x_4}{N}
	\end{cases}
	$$	
	is $\sim \frac1{N^2}.$
	
	Moreover, $\psi$ maps cubes $Q$ with side length $L$ to subsets of rectangular boxes of dimensions roughly $L\times L\times \frac{L}{N}\times \frac{L}{N}$.
	
	If the cube $Q$ is centered at the origin, $\psi(Q)$ contains the rectangular box $[-o(L),o(L)]^2\times [-o(\frac{L}{N}),o(\frac{L}{N})]^2$.
\end{lem}
\begin{proof}
	Let $\phi_1(u)=u$, $\phi_2(u)={u^2}$. Then the Jacobian is	
	\begin{equation}
	\label{fe7}
	\frac{1}{N^2}\operatorname{det}\left[ \begin{array}{cccc}
	\phi_1'(t) & \phi_2'(t) & \phi_3'(t) & \phi_4'(t) \\ \phi_1'(s) & \phi_2'(s) & \phi_3'(s) & \phi_4'(s)\\ \phi_1''(t) & \phi_2''(t) & \phi_3''(t) & \phi_4''(t) \\ \phi_1''(s) & \phi_2''(s) & \phi_3''(s) & \phi_4''(s)
	\end{array} \right].
	\end{equation}
	Note that $$\operatorname{det}\left[ \begin{array}{cccc}
	\phi_1'(t) & \phi_2'(t) & \phi_3'(t) & \phi_4'(t) \\ \phi_1'(s) & \phi_2'(s) & \phi_3'(s) & \phi_4'(s)\\ \phi_1''(t) & \phi_2''(t) & \phi_3''(t) & \phi_4''(t) \\ \phi_1''(s) & \phi_2''(s) & \phi_3''(s) & \phi_4''(s)
	\end{array} \right]=\lim_{\epsilon\to 0}\frac1{\epsilon^2}\operatorname{det}\left[ \begin{array}{cccc}
	\phi_1'(t) & \phi_2'(t) & \phi_3'(t) & \phi_4'(t) \\ \phi_1'(s) & \phi_2'(s) & \phi_3'(s) & \phi_4'(s)\\ \phi_1'(t+\epsilon) & \phi_2'(t+\epsilon) & \phi_3'(t+\epsilon) & \phi_4'(t+\epsilon) \\ \phi_1'(s+\epsilon) & \phi_2'(s+\epsilon) & \phi_3'(s+\epsilon) & \phi_4'(s+\epsilon)
	\end{array} \right].$$
	A generalization of the Mean-Value Theorem (see \cite{PS}, Voll II, part V, Chap 1, No. 95) guarantees that
	$$\operatorname{det}\left[ \begin{array}{cccc}
	\phi_1'(t) & \phi_2'(t) & \phi_3'(t) & \phi_4'(t) \\ \phi_1'(s) & \phi_2'(s) & \phi_3'(s) & \phi_4'(s)\\ \phi_1'(t+\epsilon) & \phi_2'(t+\epsilon) & \phi_3'(t+\epsilon) & \phi_4'(t+\epsilon) \\ \phi_1'(s+\epsilon) & \phi_2'(s+\epsilon) & \phi_3'(s+\epsilon) & \phi_4'(s+\epsilon)
	\end{array} \right]=$$
	\begin{align*}
	&=\epsilon^2(t-s)^2(t+\epsilon-s)(s+\epsilon-t)\operatorname{det}\left[ \begin{array}{cccc}
	\phi_1'(\tau_1) & \phi_2'(\tau_1) & \phi_3'(\tau_1) & \phi_4'(\tau_1) \\ \phi_1''(\tau_2) & \phi_2''(\tau_2) & \phi_3''(\tau_2) & \phi_4''(\tau_2)\\ \phi_1'''(\tau_3) & \phi_2'''(\tau_3) & \phi_3'''(\tau_3) & \phi_4'''(\tau_3) \\ \phi_1''''(\tau_4) & \phi_2''''(\tau_4) & \phi_3''''(\tau_4) & \phi_4''''(\tau_4)
	\end{array} \right]\\&=\epsilon^2(t-s)^2(t+\epsilon-s)(s+\epsilon-t)\operatorname{det}\left[ \begin{array}{cccc}
	\phi_3'''(\tau_3) & \phi_4'''(\tau_3) \\ \phi_3''''(\tau_4) & \phi_4''''(\tau_4)
	\end{array} \right]
	\end{align*}
	for some $\tau_i\in[\frac12,1]$ depending on $t,s,\epsilon$. The conclusion follows by letting $\epsilon\to 0$ and using \eqref{2}.
	
	The second statement is immediate. To prove the last one, assume
	$$y\in [-cL,cL]^2\times [-c\frac{L}{N},c\frac{L}{N}]^2,$$
	for some small enough $c$, independent of $N$. We need to prove that $y=\psi(x)$ for some $x\in Q$. This can be seen by solving for $x$. For example,
	$$x_1\sim N^2\operatorname{det}\left[ \begin{array}{cccc}
	y_1 & \phi_2'(t) & \phi_3'(t) & \phi_4'(t) \\ y_2 & \phi_2'(s) & \phi_3'(s) & \phi_4'(s)\\ y_3 & \frac{\phi_2''(t)}{N} & \frac{\phi_3''(t)}{N} & \frac{\phi_4''(t)}{N} \\ y_4 & \frac{\phi_2''(s)}{N} & \frac{\phi_3''(s)}{N} & \frac{\phi_4''(s)}{N}
	\end{array} \right].$$
	This and \eqref{1} show that
	$$|x_1|\lesssim N^2(\frac{|y_1|+|y_2|}{N^2}+\frac{|y_3|+|y_4|}{N}).$$
    The same inequality holds for all $x_i$, which proves the desired statement. 	

	\end{proof}
	\bigskip
	
	\begin{lem}
		\label{22}Let $\gamma$  be a Schwartz function supported on $[-2,2]$. Define the smooth Weyl sums for $u,w,v\in\R$
		$$G(u,w,v)=\sum_{k\in \Z}\gamma(k/M)e(ku+k^2w+k^3v).$$
		
		Let $1\le  b\le q\le M$ with  $(b,q)=1$. Assume that $\dist(w-\frac{b}{q},\Z):=\varphi\le \frac{1}{qM}$ and that $|v|\lesssim \frac1{M^3}$. Then  for each $\epsilon>0$ we have
		$$|G(u,w,v)|\lesssim_\epsilon \frac{M^\epsilon}{q^{1/2}}\min\{M,\frac1{\varphi^{1/2}}\}$$
		if  $$u\in \Mc=\bigcup_{m\in\Z}[\frac{m}{q}-\varphi M^{1+\epsilon},\frac{m}{q}+\varphi M^{1+\epsilon}]$$
		and
		$$|G(u,w,v)|\lesssim_\epsilon M^{-100}$$
		if  $$u\not\in \Mc.$$
	\end{lem}
	\begin{proof}
		Invoking periodicity, we may assume that  $w=\frac{b}q+\varphi$ with $|\varphi|\le \frac{1}{Mq}$. Using the representation $k=rq+k_1$, $0\le k_1\le q-1$ and the Poisson summation formula we get
		$$G(u,w,v)=\sum_{k_1=0}^{q-1}e(k_1^2b/q)\sum_{r\in\Z}\gamma(\frac{k_1+rq}{M})e((rq+k_1)u+(rq+k_1)^2\varphi+(rq+k_1)^3v)$$
		$$=\sum_{m\in\Z}\left[\frac1q\sum_{k_1=0}^{q-1}e(k_1^2b/q-k_1m/q)\right]\left[\int_\R\gamma(y/M)e((u+\frac{m}q)y+\varphi y^2+vy^3)dy\right]$$
		\begin{equation}
		\label{e2}
		=\sum_{m\in\Z}S(b,m,q)J(u,v,\varphi,m,q)
		\end{equation}
		where
		$$S(b,m,q)=\frac1q\sum_{k=0}^{q-1}e(k^2b/q-km/q)$$
		\begin{align*}
		J(u,v,\varphi,m,q)&=\int_{\R}\gamma(y/M)e((u+\frac{m}q)y+\varphi y^2+vy^3)dy\\&=M\int_{\R}\gamma(z)e(M(u+\frac{m}q)z+\varphi M^2 z^2+vM^3z^3)dz.
		\end{align*}
		If $|\varphi|\lesssim \frac1{M^2}$, we are content with the bound $|J(u,v,\varphi,m,q)|\lesssim M$.
		
		Assume now that $|\varphi|\gg \frac1{M^2}$.  The classical van der Corput estimate (second derivative test) reads
		$$|\int_{\R}\gamma(z)e(Az+Bz^2+Cz^3)dz|\lesssim |B|^{-1/2},$$
		if $|B|\gg |C|$. In our case $|B|=|\varphi| M^2\gg 1\gtrsim |vM^3|=|C|$.
		In either case we get
		$$|J(u,v,\varphi,m,q)|\lesssim \min\{M,|\varphi|^{-1/2}\}.$$
		On the other hand, repeated integration by parts (first derivative test) shows that for each $\alpha>0$
		$$|J(u,v,\varphi,m,q)|\lesssim_{\alpha}\frac1{A^\alpha} $$
		when $|A|=M|u+\frac{m}{q}|\ge M^{\epsilon}\varphi M^2$. Thus, when $u\in\Mc$, only $O(M^{\epsilon})$ values of $m$ will have a non-negligible contribution to the sum, while if $u\not\in\Mc$ then the contribution from all $m$ will be negligible.
		
		  Combining these with the classical estimate
		$$|S(b,m,q)|\lesssim \frac{1}{\sqrt{q}}$$
	    finishes the argument.
		
	\end{proof}

\end{document}